\documentclass[a4paper,12pt,thmsa]{amsart}
 \usepackage{amssymb}
  \usepackage{amsthm}
   \usepackage{mathtools}
 \def\qed{\hfill$\Box$\medskip}
 \usepackage{bbm}
 \usepackage{enumitem}
 \usepackage{mathrsfs}
\usepackage[english]{babel}
\usepackage[a4paper,top=1.9cm,bottom=1.8cm,left=1.8cm,right=1.8cm,marginparwidth=1.75cm]{geometry}
\theoremstyle{plain}
\newtheorem{proposition}{Proposition}
\newtheorem{theorem}{Theorem}
\newtheorem{lemma}{Lemma}
\newtheorem{example}{Example}
\theoremstyle{plain}
\newtheorem{theostar}{Theorem}

\newtheorem{lemmastar}{Lemma}

\newtheorem{propstar}{Proposition}

\newcommand{\asp}{\hyperlink{asp}{\color{black}$\mathbb{H}$}}
\newtheorem*{assumption}{Assumption}
\theoremstyle{remark}
\newtheorem{remark}{Remark}
\usepackage{amsmath}
\usepackage{graphicx}
\usepackage[colorlinks=true, allcolors=blue]{hyperref}
\newcommand{\ddr}{\mathrm{d}}
\title[Conditioning the Logistic CB process to survive from the total progeny]{Conditioning the Logistic continuous-state branching process on non-extinction via its total progeny}
\author{Cl\'ement Foucart}
\address{Cl\'ement Foucart, CMAP, Ecole Polytechnique, IP Paris, Palaiseau, France} 
\email{clement.foucart@polytechnique.edu}
\author{V{\'\i }ctor Rivero}
\address{V{\'\i }ctor Rivero, Centro de
Investigaci\'on en Matem\'aticas (CIMAT A.C.). Calle Jalisco s/n, 36240 Guanajuato, Guanajuato M\'exico. }
\email{rivero@cimat.mx}
\author{Anita Winter}
\address{Anita Winter, Fakult\"at f\"ur Mathematik, Universit\"at Duisburg-Essen, Campus Essen, Universit\"atsstra{\ss}e 2, 45132 Essen, Germany}
\email{anita.winter@uni-due.de}
\begin{document}
\maketitle

\begin{abstract}
The problem of conditioning a continuous-state branching process with quadratic competition (logistic CB process) on non-extinction is  investigated. We first establish that non-extinction is equivalent to the total progeny of the population being infinite. The conditioning we propose is then designed by requiring the total progeny to exceed arbitrarily large exponential random variables. This is related to a Doob's $h$-transform with an explicit excessive function $h$. The $h$-transformed process, i.e. the conditioned process, is shown to have a finite lifetime almost surely (it is either killed or it explodes continuously). 
When starting from positive values, the conditioned process is furthermore characterized, up to its lifetime, as the solution to a certain stochastic equation with jumps. The latter superposes the dynamics of the initial logistic CB process with an additional density-dependent immigration term. Last, it is established that the conditioned process can be starting from zero.   Key tools employed are a representation of the logistic CB process through a time-changed generalized Ornstein-Uhlenbeck process, as well as Laplace and Siegmund duality relationships with auxiliary diffusion processes.
\end{abstract}
\section{Introduction}
Branching processes conditioned on never becoming extinct are a cornerstone in the theory of branching processes. It is well known that they have the same law as certain branching processes with immigration, which is derived by size-biasing the offspring distribution. Immigration is interpreted as an immortal genealogical line, or spine, on which copies of the initial branching process are grafted, see Lyons et al. \cite{zbMATH00823397}, Duquesne \cite{zbMATH05509502} and Chen and Delmas \cite{zbMATH06111050} for some applications.

There is considerable interest in exploring such conditionings for more general stochastic population models that  account for interactions between individuals. Notably, these interactions typically disrupt the branching property, rendering the analysis of the process more challenging.

The objective of this article is to study a  conditioning for the class of one-dimensional continuous-state branching processes (CB processes) with \textit{quadratic competition}. These processes, referred to as logistic branching processes, were defined by Lambert in \cite{MR2134113} and are of particular interest in the theory of stochastic population models with self-regulation properties. 

They can be seen as random analogues of Verhulst's logistic function, see \cite{verhulst1838notice}, initially designed to model populations with finite carrying capacity, thereby preventing Malthusian growth. Heuristically, in addition to the classical reproduction dynamics, continuously in time and at constant rate, a pair of individuals is picked, and the first kills the second. The process, which tracks the total size of the continuous population, is the so-called logistic continuous-state branching (LCB) process.  

An LCB process $Z$, starting from $z\in (0,\infty)$, can be seen as the unique weak solution, up to the first explosion time, of the following 
stochastic equation  with jumps:
\begin{align}
Z_t=z+\sigma\int_0^t\sqrt{Z_s}\ddr B_s+\gamma\int_0^{t}Z_s\ddr s &+\int_0^t\int_{0}^{Z_{s-}}\!\!\int_{0}^{1}y\bar{\mathcal{M}}(\ddr s,\ddr u, \ddr y)\nonumber \\
&+\int_0^t\int_{0}^{Z_{s-}}\!\!\int_{1}^{\infty}y\mathcal{M}(\ddr s,\ddr u, \ddr y)
-\frac{c}{2}\int_0^{t}Z^2_s\ddr s, \ t\in [0,\infty), 
\label{SDELCB}
\end{align}
where $c,\sigma\geq 0$, $\gamma \in \mathbb{R}$, $B$ is a Brownian motion and $\mathcal{M}$ is an independent Poisson random measure (PRM) with intensity $\ddr s\ddr u\pi(\ddr y)$ with $\pi$ a Lévy measure such that $\int_{0}^{\infty}1\wedge y^2\pi(\ddr y)<\infty$. 
\smallskip

When $c>0$, the negative quadratic drift in \eqref{SDELCB} represents the competition. Heuristically continuously in time at constant rate $c/2$, a pair of individuals is picked and the first kills the second. When $c=0$, there is no competition, and the process $Z$ is a classical CB($\Psi$). We refer the reader to Dawson and Li \cite{DawsonLi} and \cite[Chapter 10]{zbMATH07687769} for the setting of CB processes. The stochastic equation \eqref{SDELCB} with competition has been studied in several works, including Berestycki et al. \cite[Proposition 1.1]{zbMATH06982256} and Palau and Pardo \cite[Theorem 1]{zbMATH06836271}. 
\medskip

The law of an LCB process is characterized by the competition parameter $c/2\geq 0$ and the branching mechanism $\Psi$, which takes the following Lévy-Khintchine form:
\begin{equation}\label{eq:branchingmechanism}\Psi(x):=\frac{\sigma^2}{2}x^2-\gamma x+\int_{0}^{\infty}\left(e^{-xy}-1+xy\mathbbm{1}_{\{y\leq 1\}}\right)\pi(\ddr y).
\end{equation}
%

In this article, we are interested in the regime where the process $Z$ becomes extinct almost surely, i.e., $Z_t\underset{t\rightarrow \infty}{\longrightarrow} 0$, $\mathbb{P}_z$-a.s. for all $z\in \mathbb{R}_+$. We emphasize that the framework of a continuous-state space enables sample paths to be attracted towards $0$, without necessarily reaching it. The population is sometimes said to be \textit{extinguishing}, see Kyprianou's book \cite[Chapter 12, page 345]{MR3155252}.

Our aim is to construct, using a Doob $h$-transform, a conditioning of the process $(Z,\mathbb{P}_z)$ on the singular event of non-extinction \[\mathscr{S}:=\{Z_t\underset{t\rightarrow \infty}{\longrightarrow} 0\}^c.\]
Aside from the pure diffusive case, where general methods can be applied,  this question, to the best of our knowledge, has not been addressed in the literature.
Before discussing logistic CB processes, we will briefly review what is known about classical CB processes. 

\subsubsection*{Without competition}
First of all, in the framework of classical CB processes, the event of extinction $\{Z_t\underset{t\rightarrow \infty}{\longrightarrow} 0\}$ happens almost surely if and only if the process is critical or subcritical, that is to say $\Psi'(0+)=0$ or $\Psi'(0+)>0$. We refer for instance to \cite[Chapter 12]{MR3155252}. Furthermore, a necessary and sufficient condition for the process to hit $0$ is the so-called Grey's condition, $\int^{\infty}\frac{\ddr u}{\Psi(u)}<\infty$. Under this condition, the extinction time $\zeta_0:=\inf\{t>0: Z_t=0\}$, of a (sub)-critical CB process $Z$, is finite almost surely and one has $\mathscr{S}=\{\exists t>0: Z_t=0\}^c$ a.s.. 
\smallskip


In the setting of finite-time extinction, the conditioning of the process to survive forever has been studied in Li \cite{MR1727226} and  Lambert \cite{lambert:tel-00252150,zbMATH05214036}. The conditioned process is built by forcing the extinction time to be larger than any deterministic time. 
This is the notion of $Q$-process. Heuristically, the null-set $\mathscr{S}$ is ``approached" by $\bigcap_{s>0}\{\zeta_0>s\}$ and 
the law of the $Q$-process arises as the following limit:
\[\mathbb{Q}_z(\Lambda):=\underset{s\rightarrow \infty}{\lim}\mathbb{P}_z(\Lambda\,|\zeta_0>t+s), \quad \forall\, t\geq 0, \, \forall\, \Lambda \in \mathcal{F}_t. \]
The law $\mathbb{Q}_z$ is furthermore related to $\mathbb{P}_z$ by the change of measure with respect to the martingale $(e^{\varrho t}Z_t,t\geq 0)$, where $\varrho:=\Psi'(0+)\geq 0$, that is to say:
\[\ddr \mathbb{Q}_z=e^{\varrho t}\frac{Z_t}{z}\ddr \mathbb{P}_z  \text{ on } \mathcal{F}_t,\ \forall\, t\geq 0.\]

As far as we know, the issue of conditioning   a (sub)critical continuous-state space branching process to survive when the process is only extinguishing (and hence Grey's condition does not hold), has not been investigated yet. We briefly explain how such a conditioning can be achieved by choosing another way to ``approximate" the null-set $\mathscr{S}$. The approach we are going to design for LCB processes will follow the same methodology.
\smallskip

The total progeny of the population is the  random variable \begin{equation}\label{progeny}J:=\int_{0}^{\infty}\! Z_s\ddr s.
\end{equation}
As noticed by Bingham \cite[Proposition 2.3 and Remark 2.4]{BINGHAM1976217}, see also \cite[Corollary 12.10]{MR3155252}, for any CB process $Z$, the following events are identical up to a null set:
\begin{equation}\label{identityeventintro}
\{Z_t\underset{t\rightarrow \infty}{\longrightarrow} 0\}=\{J<\infty\}.
\end{equation}
This entails that it is equivalent to condition the process not to get extinct or condition it to have an infinite progeny. We will define the conditioned process by forcing the total progeny to be infinite by exceeding arbitrarily large exponential random variables. This type of conditioning, defined in relation with exponential random variables, is well-known in the theory of L\'evy processes. Notable references are the works by Chaumont and Doney \cite{zbMATH05070601}, and Kyprianou et al. \cite{zbMATH06696066}. This approach relies solely on knowledge of Laplace transforms and proves to be an effective method in many cases. 

Let $\mathbbm{e}$ be an independent standard exponential random variable (i.e. with unit mean). The event of survival  is ``approached" by $\bigcap_{\theta>0}\{J\geq \mathbbm{e}/\theta\}$. We will see that the following limit exists:
\begin{equation}\label{conditioningintro} \mathbb{P}_z^{\uparrow}(\Lambda,t<\zeta)=\underset{\theta \rightarrow 0}{\lim}\,\mathbb{P}_z(\Lambda, t\leq \mathbbm{e}/\theta\,|J>\mathbbm{e}/\theta), \ \forall \Lambda\in \mathcal{F}_t, \forall t\geq 0,
\end{equation}
with $\zeta$ the lifetime of the process under $\mathbb{P}_z^{\uparrow}$. Moreover, the measure  $\mathbb{P}_z^{\uparrow}$ is the Doob's transform of $\mathbb{P}_z$, defined as follows
\[\mathbbm{1}_{\{t<\zeta\}}\ddr \mathbb{P}^{\uparrow}_z=\frac{Z_t}{z}\ddr \mathbb{P}_z  \text{ on } \mathcal{F}_t,\ \forall\, t\geq 0,\]
where $\zeta:=\inf\{t>0:\, Z_t\notin [0,\infty)\}$ is the lifetime of $(Z,\mathbb{P}_z^{\uparrow})$ and $\infty$ is chosen as a \textit{cemetery state}. Background on Doob's transforms is given in the forthcoming Section~\ref{sec:doobtransform}. It is worth noticing that $\mathbb{P}_z^{\uparrow}$ and $\mathbb{Q}_z$ only differ on $\mathcal{F}_t$ by the mass $e^{\rho t}$. Aside from the killing term, they thus both coincide.
\subsubsection*{With competition}
The longterm behaviors of the LCB process have been investigated in \cite{MR3940763}. We gather here the results we need for our purpose. We start by explaining under which conditions, extinction occurs almost-surely.

In contrast with the previous setting, the quadratic competition term ``$-\frac{c}{2}Z_t^2\ddr t$", in general may push the population with supercritical branching dynamics towards extinction. We will therefore also consider branching mechanisms $\Psi$ such that $\Psi'(0+)\in [-\infty,0)$. 
\smallskip


According to \cite[Theorem 3.9-(1)]{MR3940763}, the extinction of the LCB$(\Psi,c/2)$ process happens $\mathbb{P}_z$-almost surely, meaning the event of survival $\mathscr{S}$ is a $\mathbb{P}_z$-null set, if and only if the following condition holds
\begin{center}
\hypertarget{asp}{$\mathbb{H}$:} $\qquad \qquad \mathcal{E}:=\int_0^{x_0}\!\tfrac{1}{u}e^{\int_u^{x_{0}}\frac{2\Psi(v)}{cv}\ddr v}\ddr u=\infty \text{, for some}\, x_0>0, \text{ and } \Psi(\infty):=\underset{x\rightarrow \infty}{\lim} \Psi(x)=\infty.$
\end{center}
\smallskip

The first integral test, $\mathcal{E}=\infty$, guarantees that the LCB process $Z$ cannot explode, i.e. it does not hit $\infty$ in finite time, see \cite[Theorem 3.1]{MR3940763}. The value $x_0>0$ does not play any role in the finiteness of the integral. We shall therefore fix it in $(0,\infty)$. The condition $\mathcal{E}=\infty$ implies moreover that the LCB process has its boundary $\infty$ as an instantaneous entrance, see \cite[Theorem 3.3]{MR3940763}.  Namely, besides that the process cannot explode, it can start from $\infty$ and will leave it instantaneously.
A sufficient condition for $\mathcal{E}=\infty$, is that $\int_{0+}\frac{|\Psi(v)|}{v}\ddr v<\infty$. The latter integral condition is equivalent to the following moment condition $\int^{\infty}\log(y) \pi(\ddr y)<\infty$, see e.g. 
\cite[Proposition 3.13]{MR3940763}. 
\medskip

The second condition, $\Psi(\infty)=\infty$, is equivalent to $\Psi(x)\geq 0$ for sufficiently large $x$. This condition ensures that the pure CB process with mechanism $\Psi$ is not immortal, meaning it is not going to $\infty$ almost surely. This prevents a potential phenomenon of stationarity from occurring in the LCB process, as detailed in \cite[Theorem 3.7]{MR3940763}. To say it in another way, without competition, under the condition $\Psi(\infty)=\infty$, the CB$(\Psi)$ will tend to $0$ with positive probability, see \cite[Theorem 12.5]{MR3155252}. One says that the boundary $0$ is \textit{attracting}. Heuristically, with the additional negative competition drift, this probability becomes one, when the process cannot explode. 
\medskip

As a matter of fact, Grey's condition is also necessary and sufficient for the boundary $0$ of an LCB process to be accessible; see \cite[Theorem 3.9]{MR3940763} and \cite[Theorem 3.5]{MR2134113} (with the finite log-moment assumption). In other words, the competition solely does not cause the population to die out \textit{in finite time}. It prevents however the supercritical growth when $\mathcal{E}=\infty$. 
\medskip

We summarize the classification in Table \ref{classificationZ}.

\begin{table}[h!]
\begin{center}
\begin{tabular}{|c|c|}
\hline
Condition &  Boundary  of $Z$ \\
\hline
$\mathcal{E}=\infty$ & $\infty$  is an entrance  \\
\hline
$\Psi(\infty)=\infty$  & $0$  is attracting with positive probability\\
\hline
$\mathcal{E}=\infty$ and $\Psi(\infty)=\infty$  & $0$  is attracting almost surely\\
\hline
$\int^{\infty}\frac{\ddr x}{\Psi(x)}<\infty (\Longrightarrow \Psi(\infty)=\infty)$ & $0$  is accessible\\
\hline
\end{tabular}
\vspace*{2mm}
\caption{Boundaries of $Z$.}
\label{classificationZ}
\end{center}
\end{table}
\vspace*{-5mm}
%
\begin{example}[Set of examples, with or without a finite log-moment, satisfying \asp .]\
\begin{enumerate}
\item Stable mechanisms. Let $\alpha\in (1,2], \gamma\in \mathbb{R}, a>0$, \[\Psi(x):=ax^{\alpha}-\gamma x,\ \forall x\geq 0.\] It satisfies $\Psi(\infty)=\infty$ and $\Psi'(0+)=-\gamma$. Hence, $\int_{0}\frac{|\Psi(x)|}{x}\ddr x<\infty$ and \asp \ is fulfilled. Notice that Grey's condition is satisfied, so that extinction occurs by absorption at $0$ in finite time.
\smallskip
\item Neveu's mechanism. Let \[\Psi(x):=x\log x,\ \forall x\geq 0.\]  It satisfies $\Psi(\infty)=\infty$ and, albeit $\Psi'(0+)=-\infty$, one has  $\int_{0}\frac{|\Psi(x)|}{x}\ddr x<\infty$, hence \asp \ holds.  Notice that Grey's condition is not satisfied,  the LCB process, with mechanism $\Psi$, becomes thus extinct by converging towards $0$ without reaching it.
\smallskip
\item Let $\alpha\in (0,c/2]$, $a>0, \beta \in [1,2]$. Assume \[\Psi(x)\underset{x\rightarrow 0}{\sim}-\alpha/\log(1/x) \text{ and }  \Psi(x)\underset{x\rightarrow \infty}{\sim} ax^{\beta}.\] Then, $\int_{0}\frac{|\Psi(x)|}{x}\ddr x=\infty$, the log-moment is thus infinite, but the condition  $\alpha\leq c/2$ ensures that $\mathcal{E}=\infty$, see $\mathcal{E}$ in \asp. Furthermore, since $d,\beta>0$, $\Psi(\infty)=\infty$ and \asp \ is fulfilled. We refer to \cite[Example 3.14]{foucart2021local} for the form of the associated Lévy measure $\pi$. See also Duhalde et al. \cite[Section 7.3]{MR3264444} for other examples of Lévy measures, with infinite log-moment, verifying $\mathcal{E}=\infty$. Last, note that Grey's condition holds if and only if $\beta>1$.
\end{enumerate} 
\end{example}

An inherent difficulty in the setting with competition is that many arguments known for the branching world do not apply.  We will see however that the identity  \eqref{identityeventintro}, which relates extinction of CB processes to the event of having a finite total progeny, also holds true for logistic CBs. Next, we identify a positive excessive function $h$, i.e. such that $(h(Z_t),t\geq 0)$ is a positive supermartingale. 
The latter will be related to the conditioning defined by forcing the progeny $J$ to be greater than arbitrarily large exponential random variables, in the same spirit as explained previously for the CB processes in \eqref{conditioningintro}. More precisely, we will build a probability measure $\mathbb{P}^{\uparrow}_z$ such that
\begin{equation}\label{eq:Pupintro}\mathbb{P}^{\uparrow}_z(\Lambda,t< \zeta)=\underset{\theta \rightarrow 0}{\lim}\, \mathbb{P}_z\!\left(\Lambda, t\leq \mathbbm{e}/\theta\,\big\lvert J\geq \mathbbm{e}/\theta\right)=\mathbb{E}_z\left(\frac{h(Z_t)}{h(z)}\mathbbm{1}_{\Lambda}\right),  \ \forall \Lambda\in \mathcal{F}_{t},\ t\geq 0, 
\end{equation} with $\mathbbm{e}$ being an independent standard exponential random variable. It is worth noticing that the function $h$ will have an explicit form in terms of $\Psi$ and $c$. 
\smallskip
\subsubsection*{The role of duality}
The key tool in order to find the excessive function $h$ lies in some duality relationships satisfied by the LCB process. We briefly explain them. More details are given in Section \ref{sec:duality}. 
\smallskip

Any LCB process $Z$ admits a \textit{Laplace dual} process $U:=(U_t,t\geq 0)$. Specifically, for all $x,z\in (0,\infty)$ and $t\geq 0$, \begin{equation}\label{laplacedualintro}
\mathbb{E}_z(e^{-xZ_t})=\mathbb{E}_x(e^{-zU_t}),\end{equation}
where $U$ is the diffusion, weak solution to \[\ddr U_t=\sqrt{cU_t}\ddr B_t-\Psi(U_t), U_0=x,\]
for some Brownian motion $B$. The relationship \eqref{laplacedualintro} has been established in several works, including Hutzenthaler and Wakolbinger \cite{MR2308333}, Hutzenthaler and Alkemper \cite{MR2320823} and Greven et al. \cite{greven2015multitypespatialbranchingmodels} for the setting of \textit{spatial} logistic Feller diffusions, whose branching mechanism is $\Psi(x):=\frac{\sigma^2}{2}x^{2}-\gamma x$. The setting without spatial component and with a general branching mechanism $\Psi$, i.e. of the form \eqref{eq:branchingmechanism}, is studied in \cite{MR3940763}. 
\smallskip

The diffusion $U$, which is stochastically monotone, meaning that for any $y>0$, the map $x\mapsto \mathbb{P}_x(U_t\geq y)$ is non-decreasing, also gives rise to a certain dual process $V$. For a general statement, refer to Siegmund \cite{MR0431386}, and for the framework under consideration, see Foucart \cite[Section 6]{foucart2021local}. Specifically, for all $x,y\in (0,\infty)$ and $t\geq 0$, there is a process $V$, called \textit{Siegmund dual} of $U$, such that
$$\mathbb{P}_x(U_t\geq y)=\mathbb{P}_y(x\geq V_t),$$
where, to alleviate the notation, we keep denoting by $\mathbb{P}_{\mathrm{a}}$ the law of each process starting at $\mathrm{a}$. The process $(V,\mathbb{P}_y)$ is the diffusion, weak solution to
\[\ddr V_t=\sqrt{cV_t}\ddr B_t+\big(c/2+\Psi(V_t)\big)\ddr t, \ V_0=y,\]
with $B$ some Brownian motion.
\smallskip

By combining the two duality relationships, we will see in the forthcoming Proposition \ref{lem:joiningduals1}, that for all $t\geq 0$, $x,z\in (0,\infty)$,
\[\mathbb{E}_z(e^{-xZ_t})=\int_{0}^{\infty}ze^{-zy}\mathbb{P}_{y}(V_t>x)\ddr y.\]
We call $V$ the \textit{bidual} process of $Z$. This relationship will enable us to define the excessive function $h$ for the LCB process $Z$, in an explicit manner, by utilizing the scale function $S$ of $V$.  The diffusion setting provides many tools and will facilitate a deep study of the excessive function of $Z$. Instead of studying the LCB process $Z$ directly, we will first analyze $V$ and its Doob's transform associated to $S$. The information obtained will then be transferred to the Doob's $h$-transform of $Z$.
\smallskip

We will then study the process under the new measure $\mathbb{P}^{\uparrow}_z$, establish the identity in \eqref{eq:Pupintro}, and investigate, among other things, the stochastic equation solved by it, if any. In a reminiscent way as in the CB case, where a dichotomy occurs depending on whether the mechanism is critical or subcritical, two different settings will arise according to whether the branching Lévy measure has a finite log-moment or not.
\section{Main results}
We start by a statement for the pure CB process. This can be seen as a benchmark result in order to understand the next to come, in which competition is considered. Recall that we denote the total progeny  by $J=\int_{0}^{\infty}Z_s\ddr s$. 

As explained in the introduction, extinction occurs if and only if the total progeny is finite, see \eqref{identityeventintro}, and we will condition the CB process to never converge towards $0$ by forcing $J$ to be greater than arbitrarily large exponential random variables, see \eqref{conditioningintro}. By doing so, we recover, \textit{aside from the killing}, the classical $Q$-process of the CB process, as studied in \cite{MR1727226} and \cite{zbMATH05214036}. It is however important to note that, unlike in those works, Grey's condition is not assumed here. The process may thus only be extinguishing.
\begin{theorem}[No competition] \label{thmCSBP} For all $z\in (0,\infty)$, let $\mathbb{P}_z$ be the law of the CB process $Z$ starting from $z$. Assume the branching mechanism $\Psi$ (sub)-critical, i.e. $\varrho=\Psi'(0+)\geq 0$ (equivalently $J<\infty$, $\mathbb{P}_z$-a.s.).   
\begin{enumerate}[label=\roman*)]
\item The process $(Z_t,t\geq 0)$ is a $\mathbb{P}_z$-martingale in the critical case, i.e. $\varrho=0$. It is a $\mathbb{P}_z$-supermartingale in the subcritical case, i.e. $\varrho>0$.
\smallskip 

Define the (sub)-probability measure $\mathbb{P}^{\uparrow}_z$ as follows. Set
\[\mathbb{P}^{\uparrow}_z(\Lambda,t<\zeta):=\mathbb{E}_z\left(\frac{Z_t}{z}\mathbbm{1}_{\Lambda}\right),\quad \forall \, \Lambda \in \mathcal{F}_t,\  \forall\, t\geq 0,
\]
with $\zeta:=\inf\{t>0: Z_t\notin [0,\infty)\}$ the lifetime of $(Z,\mathbb{P}_z^{\uparrow})$.
\medskip

Two possible scenarios arise.

\smallskip

\begin{itemize}
\item In the critical case, $\mathbb{E}_z(J)=\infty$ and the lifetime $\zeta$ is infinite $\mathbb{P}_z^{\uparrow}$-a.s..
\smallskip
\item In the subcritical case, i.e. $\varrho>0$,  $\mathbb{E}_z(J)<\infty$, the lifetime $\zeta$ is finite $\mathbb{P}_z^{\uparrow}$-a.s. and has an exponential law with parameter $\varrho$.
\smallskip
\end{itemize}
\item  For all $z\in (0,\infty)$, the (sub)-probability measure $\mathbb{P}_z^{\uparrow}$ arises as the following conditioning: 
\[\mathbb{P}^{\uparrow}_z(\Lambda,t<\zeta)=\underset{\theta \rightarrow 0}{\lim}\, \mathbb{P}_z(\Lambda, t\leq \mathbbm{e}/\theta\,|J>\mathbbm{e}/\theta), \ \forall \Lambda \in \mathcal{F}_t,\ \forall t\geq 0,\]
where $\mathbbm{e}$ denotes a standard exponential random variable independent from $(Z,\mathbb{P}_z)$.

\smallskip

\item
For any $z\in (0,\infty)$, the process $(Z,\mathbb{P}_z^{\uparrow})$ has the same law as the unique weak solution to the following stochastic equation, killed at an exponential random variable $\zeta$ with parameter $\varrho$:
\begin{align}
Z_t=z&+\sigma\int_0^t\sqrt{Z_s}\ddr B_s-\varrho\int_0^{t}Z_s\ddr s +\int_0^t\int_{0}^{Z_{s-}}\!\!\int_{0}^{\infty}y\bar{\mathcal{M}}(\ddr s,\ddr u, \ddr y)
\label{conditioned partcsbp} \\
&+\sigma^2t+\int_0^t\!\!\int_{(0,\infty)}y\mathcal{N}(\ddr s, \ddr y),\ t<\zeta,\nonumber
\end{align}
with $B$ a Brownian motion, $\mathcal{M}$ and $\mathcal{N}$ two Poisson random measures, respectively on $[0,\infty)^3$ and $[0,\infty)^2$, with intensity $\ddr s\, \ddr u\, \pi(\ddr y)$ and $\ddr s\, y\pi(\ddr y)$ and everything is mutually independent.
\end{enumerate}
\end{theorem}
\begin{remark}
At the level of the progeny, the distinction between the critical and subcritical cases can be understood through the expected value of the progeny $\mathbb{E}_z(J)$. In the critical case $\rho=0$, $\mathbb{E}_z(J)=\infty$. Conversely, $\mathbb{E}_z(J)<\infty$ when $\rho>0$. In the subcritical case, forcing the progeny to be infinite imposes such a significant distortion on the CB process that roughly, by pulling upwards the path, so to avoid it heading down to zero, we end up ``breaking" it. This is formalized by killing the conditioned process.
In the critical case, the mean of the progeny being infinite, the conditioning is less stringent.
\end{remark}
\begin{remark}
We recognize in \eqref{conditioned partcsbp} the stochastic equation solved by a continuous-state branching process with immigration (CBI) with branching mechanism $\Psi$ and immigration $\Psi'$, see \cite[Chapter 10]{zbMATH07687769}. Specifically, the last line in \eqref{conditioned partcsbp} is a subordinator with Laplace exponent $\Psi'$, representing the immigration term (with a killing term when $\varrho=\Psi'(0+)>0$).
Theorem \ref{thmCSBP} applies to the peculiar case of (sub)critical CB processes whose mechanisms are of finite variation type, i.e. $\Psi'(\infty)<\infty$, equivalently there is no diffusive part, namely $\sigma=0$ in \eqref{eq:branchingmechanism}, and the Lévy measure $\pi$ is such that $\int_0^\infty 1\wedge y\, \pi(\ddr y)<\infty$. The (sub)critical mechanism, in this setting, takes the form
\begin{equation}\label{psifinitevar} \Psi(x)=bx-\int_{0}^{\infty}(1-e^{-xy})\pi(\ddr y),
\end{equation}
with $b=\Psi'(\infty)>0$ and $\Psi'(0+)=b-\int_{0}^{\infty}y\pi(\ddr y)\geq 0$. 
Pakes and Trajtsman \cite{zbMATH03896728} have established some conditional limit theorems in this setting. 
\end{remark}
\begin{proposition}\label{propcbIfrom0} Assume $\Psi'(0+)\geq 0$. The following convergence holds in Skorokhod's sense, $$\mathbb{P}_z^{\uparrow}\underset{z\rightarrow 0+}{\Longrightarrow} \mathbb{P}_0^{\uparrow},$$ for some probability measure $\mathbb{P}_0^{\uparrow}$ characterized as the weak solution to \eqref{conditioned partcsbp} with $z=0$. 
\end{proposition}
\begin{remark}\label{rem:holdingcb} The process $(Z,\mathbb{P}^{\uparrow}_0)$ has the same law as a CBI$(\Psi,\Psi')$ starting from $0$. When the measure on $(0,\infty)$, $\nu(\ddr y):=y\, \pi(\ddr y)$ is finite and $\sigma=0$ (i.e. $\Psi$ takes the form \eqref{psifinitevar} and $b=\Psi'(\infty)<\infty$) the subordinator with Laplace exponent $\Psi'$ is a compound Poisson process and we see that $(Z,\mathbb{P}^{\uparrow}_0)$ stays an exponential time at $0$ with parameter $b=\Psi'(\infty)$, i.e. $0$ is an holding point. Conversely, if $\Psi'(\infty)=\infty$ the process leaves $0$ instantaneously since the first jump from the immigration occurs immediately.
\end{remark} 
From now on, we focus on the setting with competition, i.e. $c>0$ and we work with the hypothesis \asp. 
\begin{assumption}[\asp]
\[\mathcal{E}=\infty \text{ and } \Psi(\infty)=\infty.\]
\end{assumption}
\noindent As discussed in the Introduction, and further explained in Section \ref{sec:preliminariesLCB}, condition \asp \ is both  necessary and sufficient to ensure that the extinction event $\{Z_t\underset{t\rightarrow \infty}{\longrightarrow} 0\}$ occurs almost surely.
\smallskip

Our starting point is the following identity between the event of having a population with finite total progeny and extinction. This generalizes Bingham's result, see \cite{BINGHAM1976217}, to the case with competition. 
\begin{proposition}[Total progeny and extinction]\label{prop1} Assume \asp . For any $z\in (0,\infty)$, up to a $\mathbb{P}_z$-null set, one has the identity
\[\{Z_t\underset{t\rightarrow \infty}{\longrightarrow} 0\}=\{J<\infty\}.\]
Furthermore, for all $z\in (0,\infty)$,
\[
\mathbb{E}_z(J)<\infty \text{ if and only if } \int^{\infty}\log y\, \pi(\ddr y)<\infty.
\]
\end{proposition}
We now identify a specific positive excessive function for the LCB process. It will have Bernstein's form, see e.g. Schilling et al.'s book \cite[Theorem 3.2]{zbMATH06059089}, and we shall see that the supermartingale it induces, is a local martingale if and only if $J$ has an infinite first moment, $\mathbb{E}_z(J)=\infty$.
\begin{theorem}\label{thm1}
Assume \asp .
\begin{enumerate}[label=\roman*)]
\item 
Fix an $x_0>0$ and define the measure on $(0,\infty)$, \begin{equation}\label{eq:s}s(\ddr x):=\frac{1}{x}e^{-\int_{x_0}^{x}\frac{2\Psi(u)}{cu}\ddr u}\ddr x.
\end{equation} The following function, of Bernstein form, is well-defined, \begin{equation}\label{hintro}
\forall z\in [0,\infty),\ h(z):=\int_0^{\infty}(1-e^{-xz})s(\ddr x).
\end{equation} Furthermore, $h$ is a positive increasing  function, continuously twice differentiable on $(0,\infty)$,  such that
\begin{center}
$h(0)=0$, $\underset{z\rightarrow \infty}{\lim} h(z)=\infty$, $h'(0)<\infty$ and $\int^{\infty}h(y)\pi(\ddr y)<\infty$. 
\end{center}
\item Under $\mathbb{P}_z$, the process 
$\big(h(Z_t),t\geq 0\big)$
is a supermartingale, moreover it is
\smallskip
\begin{itemize}
\item a strict supermartingale (i.e. it is not a local martingale) when $\int^{\infty}\log y\, \pi(\ddr y)<\infty$,
\item a strict local martingale (i.e. it is not a martingale) when $\int^{\infty}\log y\, \pi(\ddr y)=\infty$.
\end{itemize} 
\end{enumerate}
\end{theorem}
We introduce the following Doob's transform induced by the supermartingale $\big(\frac{h(Z_t)}{h(z)},t\geq 0\big)$, $$\mathbbm{1}_{\{t<\zeta\}}\ddr \mathbb{P}^{\uparrow}_z:=\frac{h(Z_{t})}{h(z)}\ddr \mathbb{P}_z,\  \text{ on } \mathcal{F}_{t},\ \forall t\geq 0\text{ and } z>0,$$ 
where $\zeta:=\inf\{t>0:\, Z_t\notin [0,\infty)\}$ is the lifetime of $(Z,\mathbb{P}_z^{\uparrow})$, $\infty$ is taken as the cemetery state and on $\{\zeta<\infty\}$, we set $Z_{\zeta+t}=\infty$  for all $t\geq 0$, $\mathbb{P}_z^{\uparrow}$-a.s..
\medskip

We study below the new measure $\mathbb{P}_z^{\uparrow}$ and start by the global behavior of the paths of $(Z,\mathbb{P}_z^{\uparrow})$.
\begin{proposition}\label{prop:infimum} Assume \asp . For any $z>a\geq 0$,
\[\mathbb{P}_z^{\uparrow}\big(\underset{0\leq s<\zeta}{\inf}Z_s\leq a\big)=\frac{h(a)}{h(z)}.\]
In particular, $\underset{0\leq t<\zeta}{\inf} Z_t>0$, $\mathbb{P}_z^{\uparrow}$-a.s. for all $z>0$.
\end{proposition}
\begin{theorem} \label{mainthm}  Assume \asp . 
\begin{enumerate}[label=\roman*)]
\item For all $z>0$, the process $Z$ under $\mathbb{P}^{\uparrow}_z$ is a $(0,\infty]$-valued Feller process. Moreover, its lifetime $\zeta$ is finite $\mathbb{P}_z^{\uparrow}$-almost surely and
\smallskip
\begin{enumerate}
\item If $\int^{\infty}\log(y)\pi(\ddr y)<\infty$, the process $Z$ under $\mathbb{P}_z^{\uparrow}$  is killed (it is sent to the cemetery point $\infty$) with positive probability, that is to say $\mathbb{P}^{\uparrow}_z(Z_{\zeta-}<\infty)>0$ for all $z>0$.
\smallskip
\item If $\int^{\infty}\log(y)\pi(\ddr y)=\infty$, the process $Z$ under $\mathbb{P}_z^{\uparrow}$ explodes continuously at its lifetime almost surely, that is to say $Z_{\zeta-}=\infty$, $\mathbb{P}^{\uparrow}_z$-a.s.
\end{enumerate}
\smallskip
\item 
In any case, for all $z>0$, the law $\mathbb{P}^{\uparrow}_z$ arises as the following limit,
\[\mathbb{P}^{\uparrow}_z(\Lambda,t< \zeta)=\underset{\theta \rightarrow 0}{\lim}\, \mathbb{P}_z\!\left(\Lambda, t\leq \mathbbm{e}/\theta\,\big\lvert J\geq \mathbbm{e}/\theta\right),\quad \forall \, \Lambda\in \mathcal{F}_{t},\ \forall \, t\geq 0, \] 
where $\mathbbm{e}$ is an independent standard exponential random variable.
\end{enumerate}
\end{theorem}
\begin{remark} In a similar fashion to subcritical CB processes, we see that when there is a log-moment, conditioning the progeny to be large  is so demanding that it ``breaks the sample paths of the process $(Z,\mathbb{P}^{\uparrow})$", causing it to be killed.
\end{remark}
\smallskip
We now link the process $(Z,\mathbb{P}_z^{\uparrow})$ for $z>0$ to a certain stochastic equation in which two dynamics occur: that of the initial LCB process and a certain density-dependent immigration.   

\medskip

Recall the function $h$ in \eqref{hintro}. Set $\ell:=e^{\int_{0}^{x_0}\frac{2\Psi(u)}{cu}\ddr u}\in [0,\infty)$ and define, for all $z>0$ and $y>0$,
\begin{equation}\label{eq:bphik}
b(z):=z\frac{h'(z)}{h(z)},\quad q(z,y):=
\frac{z}{h(z)}\big(h(z+y)-h(z)\big) \text{ and } k(z):=\frac{c\ell}{2}\frac{z}{h(z)}.
\end{equation}
Note that $\ell=0$ if and only if $\int_0\frac{\Psi(x)}{x}\ddr x=-\infty$, which is equivalent to $\int^{\infty}\log y\, \pi(\ddr y)=\infty$. 


\begin{theorem}\label{thm:SDEZup}
Assume \asp . For any $z\in (0,\infty)$, the process $(Z,\mathbb{P}^{\uparrow}_z)$ has the same law as the weak solution to  Equation \eqref{SDEPuparrow} below, killed at time
$\zeta:=\inf\{t>0: \int_{0}^{t}k(Z_s)\ddr s\geq \mathbbm{e}\}$.
\begin{align}\label{SDEPuparrow}
Z_t=z&+\sigma\int_0^t\sqrt{Z_s}\ddr B_s+\gamma\int_0^{t}Z_s\ddr s+\int_0^t\int_{0}^{Z_{s-}}\!\!\int_{0}^{1}y\bar{\mathcal{M}}(\ddr s,\ddr u, \ddr y) \\
&+\int_0^t\int_{0}^{Z_{s-}}\!\!\int_{1}^{\infty}y\mathcal{M}(\ddr s,\ddr u, \ddr y) -\frac{c}{2}\int_0^{t}Z^2_s\ddr s \nonumber\\
&+\sigma^2\int_{0}^{t}b(Z_s)\ddr s+\int_0^t\int_{0}^{q(Z_{s-},y)}\!\!\int_0^{\infty}y\mathcal{N}(\ddr s,\ddr u, \ddr y), \ t<\zeta, \nonumber               
\end{align}
where
$\mathbbm{e}$ is a standard exponential random variable, $B$ is a Brownian motion, $\mathcal{M}$ and $\mathcal{N}$ are two Poisson random measures with same intensity $\ddr s\,\ddr u\, \pi(\ddr y)$, everything being mutually independent.
\end{theorem}
\begin{remark}
Note that $h$ is concave and since $h(0)=0$, one has $h'(z)\leq h(z)/z$. We then get $b(z)\leq 1$ and by applying the mean value theorem, we have 
\begin{center}
$q(z,y)\leq \frac{z}{h(z)}h'(z)y\leq y$ for any $y,z>0$.
\end{center}
The upper-bound above ensures that there is no need to compensate the Poisson random measure $\mathcal{N}$ in the last stochastic integral in \eqref{SDEPuparrow}, since  $$\int_0^1 yq(z,y)\pi(\ddr y)\leq \int_0^{1}y^2\pi(\ddr y)<\infty.$$ 
\end{remark}
\begin{remark}
The density-dependent immigration term in the last line of \eqref{SDEPuparrow} should be compared to the immigration subordinator in the setting without competition. Roughly speaking, replacing the excessive function $h$ by the identity, we get $b(z)=z$ and $q(z,y)=y$. Observe then that $\int \mathbbm{1}_{(0,y]}(u)\mathcal{N}(\ddr s,\ddr u, \ddr y)$ is  a Poisson random measure with intensity $\ddr s\,\ddr u\, y\pi(\ddr y)$, so that we recover here at a heuristic level the subordinator with drift $\sigma^2$ and Lévy measure $y\pi(\ddr y)$ as in the setting without competition, see Theorem \ref{thmCSBP}. 
\end{remark}
\begin{remark} Li has explored stochastic equations with state-dependent immigration similar to \eqref{SDEPuparrow}, see \cite{zbMATH07076056} and \cite[Chapter 12]{zbMATH07687769}. These works discuss conditions for the existence and uniqueness of strong solutions and provide a Poissonian cluster decomposition of the solution, akin to that for CBIs, see for instance \cite[Chapter 9.5]{zbMATH07687769} for the classical case without competition. 
\end{remark}
The last theorem establishes that the conditioned LCB process can start from $0$. \begin{theorem}\label{thm:entrance} Assume \asp . The following weak convergence holds in Skorokhod's sense 
\[\mathbb{P}^{\uparrow}_z\underset{z\rightarrow 0+}{\Longrightarrow}\mathbb{P}^{\uparrow}_0,\]
with $\mathbb{P}^{\uparrow}_0$ a probability measure such that \[\mathbb{P}^{\uparrow}_0\big(Z_0=0, \exists t>0: \forall s\geq t, Z_s>0\big)=1.\]
The process $(Z,\mathbb{P}^{\uparrow}_0)$ is moreover a weak solution to \eqref{SDEPuparrow} with $z=0$, where the coefficients $b,q,k$ given by \eqref{eq:bphik} are extended at $z=0$ as follows:
\begin{equation}
b(0):=1, \quad \forall y>0,\ q(0,y):=\frac{h(y)}{h'(0)}, \text{ and } k(0):=\frac{c\ell}{2}\frac{1}{h'(0)}.
\end{equation}
%
\end{theorem}
\begin{remark} In a similar fashion as for CB processes, the process $(Z,\mathbb{P}_0^{\uparrow})$ can be seen as encoding the lineage of an immortal individual. The point $0$ of $(Z,\mathbb{P}_0^{\uparrow})$ is instantaneous if in the density-dependent immigration, see the last line of \eqref{SDEPuparrow}, there is a positive drift, i.e. $\sigma>0$, or if the first jump occurs immediately. According to the jump intensity, this happens if and only if $\int_{0}h(y)\pi(\ddr y)=\infty$.  Since $h(y)\sim h'(0)y$ as $y\to 0$, we see that this is equivalent to $\Psi'(\infty)=\infty$. Hence similarly as for the CB process, see Remark \ref{rem:holdingcb}, $0$ is instantaneous (respectively an holding point) if and only if  $\Psi'(\infty)=\infty$ (respectively $\Psi'(\infty)<\infty$).
\end{remark}
\noindent \textbf{Organization of the paper}. In the preliminaries, we first give some general background on Doob's $h$-transforms. We then briefly review LCBs, and show Proposition \ref{prop1}. Next, we study the auxiliary diffusion $V$.  The proof of  Theorem \ref{thmCSBP} is deferred to Section \ref{sec:prooftheoremCSBP}. In Section \ref{sec:prooftheorem1}, we establish Theorem \ref{thm1}, particularly focussing on how to find the excessive function $h$. In Section \ref{sec:prooftheorem2}, we derive a representation of the semigroup of the conditioned process $Z^{\uparrow}$ with the help of the diffusion process $V$ conditioned to not tend towards $\infty$. We then establish Proposition \ref{prop:infimum} and Theorem \ref{mainthm}. The stochastic equation solved by the conditioned process $(Z,\mathbb{P}_z^{\uparrow})$ is addressed in Section \ref{sec:prooftheorem34}. In Section \ref{sec:proofentrance}, we establish Theorem \ref{thm:entrance}. 
\section{Preliminaries}\label{sec:preliminaries}

Denote by $[0,\infty]$, the extended positive half-line, endowed with the compact metric $d(x,y)=|e^{-x}-e^{-y}|$, for all $x,y\in [0,\infty]$, with the convention $e^{-\infty}=0$. We denote by $\Omega$ the canonical probability space, i.e. the space of functions $\omega:\mathbb{R}_+\rightarrow [0,\infty]$ which are right-continuous, admit a lifetime $\zeta:=\inf\{t>0: Z_{t}\notin [0,\infty)\}$, possibly infinite. The canonical process $(Z_t,t\geq 0)$ is defined on this probability space by $Z_t(\omega)=\omega(t)$, $\omega\in \Omega$. This space is equipped with the shift operator $\theta_t$, namely $\theta_t\omega(s)=\omega(s+t)$. The natural filtration on $\Omega$ is defined for any $t\geq 0$, by $\mathcal{F}^0_t:=\sigma(\{Z_s,0\leq s\leq t\})$. We denote by $(\mathcal{F}_t)$ the augmented right-continuous filtration. 
\subsection{Doob's $h$-transforms via supermartingales and local martingales}\label{sec:doobtransform}
We provide here background on the construction of (sub)Markovian probability laws with the help of supermartingales.  We refer the reader e.g. to Chung and Walsh's book \cite[Section 11.3, page 324]{zbMATH02208909} for background on Doob's $h$-transforms. 
\medskip

Let $(P_t)_{t\geq 0}$ be a semigroup on $[0,\infty]$. We assume that for each $z\in [0,\infty]$, there exists a right-continuous Markov process $Z:=(Z_t,t\geq 0)$ with state space $[0,\infty]$ and transition semigroup $(P_t)_{t\geq 0}$, satisfying
\begin{enumerate}[label=\roman*)]
\item For all $z\in [0,\infty]$, $\mathbb{P}_z(Z_0=z)=1$,
\item $Z$ is a strong Markov process,
\item $Z$ has an infinite lifetime, $\zeta=\infty$, $\mathbb{P}_z\text{-a.s. for all }z\in [0,\infty]$.
\end{enumerate}
\medskip

Recall that a positive function $h$ is said to be excessive for $Z$, if the process $(h(Z_t),t\geq 0)$ is a positive supermartingale, i.e. $\mathbb{E}_z(h(Z_t))\leq h(z)$ for all $t\geq 0$ and all $z\in [0,\infty)$.

\begin{theostar}[Change of measure with a supermartingale] \label{doobtransform1} Assume $h$ is an excessive function,  $0<h(z)<\infty$ for all $z\in (0,\infty)$ and $h(0)=0$, $h(\infty)=\infty$. Then,
\begin{enumerate}[label=\roman*)]
\item The family of operators defined on the $[0,\infty)$-valued bounded Borel functions by:
\[P^{\uparrow}_tf(z):=\begin{cases}
\frac{1}{h(z)}P_t(hf)(z), &\text{ if } z\in (0,\infty),\\
0, &\text{ if } z\in \{0,\infty\},
\end{cases}
\]
is a semigroup on $[0,\infty)$, possibly sub-Markovian, i.e. one has $P_t^{\uparrow}1\leq 1$ for all $t\geq 0$.
\item The process with semigroup $(P^{\uparrow}_t)_{t\geq 0}$ admits a right-continuous version and we denote by $\mathbb{P}^{\uparrow}_z$ its law on $(\Omega,\mathcal{F}^{0})$, when issued from $z>0$.  Moreover, the process $(Z,\mathbb{P}^{\uparrow}_z)$ is strongly Markovian with respect to $(\mathcal{F}_t)_{t\geq 0}$ and if $T$ is an $(\mathcal{F}_t)_{t\geq 0}$-stopping time, then for any $A\in \mathcal{F}_T$ and $z\in (0,\infty)$, 
\begin{equation}\label{equationdoob1}\mathbb{P}_z^{\uparrow}(A,T<\zeta)=\frac{1}{h(z)}\mathbbm{E}_z\left(h(Z_T)\mathbbm{1}_A\right), \text{ with } \zeta \text{ the lifetime of } (Z,\mathbb{P}_z^{\uparrow}).
\end{equation}
\end{enumerate}
\end{theostar}

The $h$-transform covers three different settings:
\begin{enumerate}[label=\roman*)]
\item If $(h(Z_t),t\geq 0)$ is a $\mathbb{P}_z$-martingale, then the process $(Z,\mathbb{P}_z^{\uparrow})$ has an infinite lifetime (it is not killed and does not explode).
\smallskip
\item If $(h(Z_t),t\geq 0)$ is a $\mathbb{P}_z$-supermartingale but not a local martingale, we say that it is a \textit{strict} supermartingale. In this case, the process $(Z,\mathbb{P}_z^{\uparrow})$ has a finite lifetime, it is moreover killed, meaning the process is sent at $\infty$ by a single jump with positive probability. Provided that there exists\footnote{this will be the case in our setting} a.s. a left limit at $\zeta$, one has $\mathbb{P}_z^{\uparrow}(Z_{\zeta-}<\infty)>0$.
\smallskip
\item If $(h(Z_t),t\geq 0)$ is a local martingale but not a martingale, we say that it is a \textit{strict} local martingale. In this case, the process has a finite lifetime, but is not killed. Explosion occurs here ``continuously", that is to say $Z_{\zeta-}=\infty$, $\mathbb{P}_z^{\uparrow}$-a.s..  
\end{enumerate}
We now provide further details on the second and third cases, as both will occur in our framework.  
\begin{lemmastar}[Strict local martingale and continuous explosion]\label{lem:strictlocalmartingale} Under the same assumptions as Theorem \ref{doobtransform1}. If $(h(Z_t),t\geq 0)$ is a supermartingale and not a martingale then $\mathbb{P}^{\uparrow}_z(\zeta<\infty)>0$. For all $n\in \mathbb{N}$, set $\zeta_n^{+}:=\inf\{t>0: Z_t>n\}\in [0,\infty]$.
\smallskip
\begin{itemize}
\item If $(h(Z_t),t\geq 0)$ is a strict local martingale, then on $\{\zeta<\infty\}$, $\zeta=\underset{n\rightarrow \infty}{\lim} \zeta_n^+$, $\mathbb{P}^{\uparrow}_z$-a.s..
\item  If $(h(Z_t),t\geq 0)$ is a strict supermartingale, then, one has for some $n\geq 1$, $\mathbb{P}_z^{\uparrow}(\zeta_n^{+}\geq \zeta)>0$. Thus, there exists an $n$ such that with positive probability, the process is killed before getting above level $n$.
\end{itemize}
\end{lemmastar}
\begin{proof}
First, by taking $A=\Omega$ and $T=t$ a deterministic time, in \eqref{equationdoob1}, we get
$\mathbb{P}_z^{\uparrow}(t<\zeta)=\mathbb{E}_z\left(\frac{h(Z_t)}{h(z)}\right)<1$,
hence, $\mathbb{P}_z^{\uparrow}(\zeta<\infty)>0$. 
\smallskip

Suppose that $(h(Z_t),t\geq 0)$ is a strict local martingale. We explain why the explosion is continuous on the event $\{\zeta<\infty\}$. We start by checking that the sequence of stopping times, $(\tau_n):=(\zeta_n^+\wedge n)$, localizes $(h(Z_t),t\geq 0)$. First, by assumption, $\zeta_n^+\rightarrow \infty$,  $\mathbb{P}_z$-a.s. as $n$ goes to $\infty$, and the same holds for $(\tau_n)$. Observe next that since $(h(Z_t),t\geq 0)$ is a supermartingale and $\tau_n$ is a bounded stopping time, one has by the optimal stopping theorem $\mathbb{E}_z\big(h(Z_{\tau_n})\big)\leq h(z)<\infty$. This entails that $h(Z_{\tau_n})\vee h(n)$ is integrable. Finally, notice that
\[\sup_{t\geq 0}h(Z_{t\wedge \tau_n})\leq h(Z_{\tau_n})\vee h(n)\text{ a.s..}\]
Since the stopped process $(h(Z_{t\wedge \tau_n}),t\geq 0)$ is a local martingale, bounded by an integrable random variable, it is actually a true martingale. By taking now $A=\Omega$ and $T=\tau_n$ in \eqref{equationdoob1}, we get for any $n\in \mathbb{N}$, 
\begin{center}
$\mathbb{P}_z^{\uparrow}(\zeta_n^+\wedge n<\zeta)=\frac{1}{h(z)}\mathbbm{E}_z\left(h(Z_{\zeta_n^+\wedge n})\right)=1.$ \end{center}
Hence, 
\begin{align*}
\mathbb{P}_z^{\uparrow}(\zeta<\infty)&=\mathbb{P}_z^{\uparrow}(\zeta_n^+\wedge n<\zeta<\infty)\\
&=\mathbb{P}_z^{\uparrow}(\zeta_n^+<n,\zeta_n^+<\zeta<\infty)+\mathbb{P}_z^{\uparrow}(\zeta_n^+\geq n,n<\zeta<\infty)\underset{n\rightarrow \infty}{\longrightarrow} \mathbb{P}_z^{\uparrow}(\zeta_m^+<\zeta<\infty, \forall m\geq 1).
\end{align*}
Therefore, denoting by $\zeta_\infty:=\underset{m\rightarrow \infty}{\lim}\! \uparrow \zeta_m^+=\inf\{t>0: Z_t=\infty\}$, we have, conditionally on the event $\{\zeta<\infty\}$,  $\zeta_\infty\leq \zeta$, $\mathbb{P}_z^{\uparrow}$-almost surely, hence $\zeta_\infty=\zeta$, $\zeta$ is thus predictable, and $\mathbb{P}_z^{\uparrow}(Z_{\zeta-}=\infty)=1$.
\smallskip

Suppose now that $(h(Z_t),t\geq 0)$ is a strict supermartingale. Then, necessarily $(\zeta_n^+\wedge n)$ is not a localizing sequence of stopping times (otherwise the process would be a local martingale). Therefore, for some $s>0,n\geq 1$, 
\begin{center}
$\mathbb{P}_z^{\uparrow}(s\wedge \zeta_n^+\wedge n<\zeta)=\frac{1}{h(z)}\mathbbm{E}_z\left(h(Z_{s\wedge \zeta_n^+\wedge n})\right)<1,$ 
\end{center}
and we finally see that $\mathbb{P}_z^{\uparrow}(\zeta_n^+\geq \zeta)>0$.\qed
\end{proof}
\begin{remark}[Almost sure killing]\label{rem:killing} In the case ii) of a strict supermartingale, we have seen that the process $(Z,\mathbb{P}_z^{\uparrow})$ is killed before explosion with positive probability. This occurs almost surely if and only if the supermartingale $\big(h(Z_{\zeta_n^+\wedge n}),n\geq 0\big)$ is $\mathbb{P}_z$-uniformly integrable. Indeed, in this scenario, recalling $h(0)=0$, and that $\mathbb{P}_z$-almost surely, $\zeta_n^+\wedge n \rightarrow \infty$ as $n$ goes to $\infty$ and $Z_t\rightarrow 0$ as $t$ goes to $\infty$, we see that $(h(Z_{\zeta_n^+\wedge n}),n\geq 0)$ converges towards $0$ in $L^1(\mathbb{P}_z)$. Therefore, we have 
\[\underset{n\rightarrow \infty}{\lim} \mathbb{P}_z^{\uparrow}(\zeta_n^+\wedge n<\zeta)=\frac{1}{h(z)}\underset{n\rightarrow \infty}{\lim}\mathbbm{E}_z\left(h(Z_{\zeta_n^+\wedge n})\right)=\frac{1}{h(z)}\mathbbm{E}_z\left(\underset{n\rightarrow \infty}{\lim}h(Z_{\zeta_n^+\wedge n})\right)=0,\]
implying that $\mathbb{P}_z^{\uparrow}(\exists n\geq 1: \zeta_n^{+}\geq \zeta)=1$.
\end{remark}
\begin{remark}[F\"ollmer measure]\label{rem:Follmermeasure} When $(h(Z_t),t\geq 0)$ is a local martingale, the measure $\mathbb{P}^{\uparrow}_z$ is called in the literature, the F\"ollmer measure. We refer the reader to F\"ollmer \cite{zbMATH03366235}, Meyer \cite{zbMATH03366236}, see also the work of It\^o and Watanabe \cite{zbMATH03227547}. Recall the stopping times $\tau_n=\zeta_n^+\wedge n$ for $n\geq 1$, and the fact that 
$(h(Z_{t\wedge \tau_n}),t\geq 0)$ is a martingale, see the proof of Lemma \ref{lem:strictlocalmartingale}. The law
$\mathbb{P}^{\uparrow}_z$ is the unique measure on $(\Omega,\mathcal{F}^0)$ with lifetime $\zeta=\underset{n\rightarrow \infty}{\lim}\tau_n$, such that $$\ddr \mathbb{P}^{\uparrow}_z=\frac{h(Z_{t\wedge \tau_n})}{h(z)}\ddr \mathbb{P}_z,\  \text{ on } \mathcal{F}_{t\wedge \tau_n},\ \forall n\geq 1, \quad \forall t\geq 0\text{ and } z>0.$$
\end{remark}
\begin{remark}[Local absolute continuity]\label{rem:nonegativejumps}\  
\begin{itemize}
\item[i)]  For any $z\in (0,\infty)$, the measure $\mathbb{P}^{\uparrow}_z$ defined by \eqref{equationdoob1}, is \textit{locally} absolutely continuous with respect to  $\mathbb{P}_z$: $\forall t\geq 0, \mathbb{P}^{\uparrow}_z|_{\mathcal{F}_t}\ll\mathbb{P}_z|_{\mathcal{F}_t}$. As a consequence, if $(Z,\mathbb{P}_z)$ is c\`adl\`ag with no negative jumps, then $(Z,\mathbb{P}^{\uparrow}_z)$ has also no negative jumps. Indeed,  for all $r>0$, set $\tau^{r}:=\inf\{t>0: \Delta Z_t:=Z_{t}-Z_{t-}<-r\}.$ This is an $(\mathcal{F}_t)$-stopping time, see e.g. Jacod and Shiryaev's book \cite[Page 7]{jacod1987limit}, and since $\{\tau^{r}<t\}$ is a $\mathbb{P}_z$-null set, one has
\begin{center} $\mathbb{P}_z^{\uparrow}(\tau^{r}<t,t<\zeta)=\frac{1}{h(z)}\mathbb{E}_z\big(h(Z_t)\mathbbm{1}_{\{\tau^{r}<t\}}\big)=0$.\end{center}
Hence, the event $\{\exists t>0: \Delta Z_t<0\}=\cup_{r\in \mathbb{Q}_+^{\star}}\{\tau^{r}<\infty\}$ is a $\mathbb{P}_z^{\uparrow}$-null set. 
\item[ii)] Notice that $\mathbb{P}^{\uparrow}_z$ may charge $\mathbb{P}_z$-null sets. For instance, by assumption here, $\mathbb{P}_z(\zeta<\infty)=0$, and according to Lemma \ref{lem:strictlocalmartingale}, $\mathbb{P}^{\uparrow}_z(\zeta<\infty)>0$ when $(h(Z_t),t\geq 0)$ is not a true martingale.
\end{itemize}
\end{remark}
\subsection{Background on LCB processes}
We gather here some fundamental properties of LCB processes. Most of those results can be found in \cite{MR3940763}. Proposition \ref{prop1} is established in forthcoming Section \ref{sec:preliminariesLCB}. We also state some basic facts on the martingale problem solved by the LCB. Those are mostly direct consequences of the stochastic equation \eqref{SDELCB} solved by the LCB process. They will play an important role in several places in the proofs.
\subsubsection{Martingale problem and Lamperti's time change for the LCB process}\label{sec:martingaleproblemLCB}
Denote by $C^2([0,\infty))$ the space of functions that are twice differentiable on $[0,\infty)$. Let $\hat{C}$ be the space of continuous functions on $[0,\infty]$, hence with a finite limit at $\infty$. Set
\begin{equation}\label{DZ}\mathcal{D}_Z:=\left\{f\in C^2([0,\infty))\cap \hat{C}: \underset{z\rightarrow \infty}{\lim} \big(z|f(z)-f(\infty)|+z|f''(z)|+(z+cz^2)|f'(z)|\big)=0\right\}.
\end{equation}
Denote by $C_c^2(0,\infty)$ the space of $C^2$ functions vanishing outside a compact set included in $(0,\infty)$. Note that $C_c^2(0,\infty)\subset \mathcal{D}_Z$.\\ 

\noindent \textbf{Martingale problem}. Let $Z$ be the solution to the stochastic equation \eqref{SDELCB}. For any $f\in C^2([0,\infty))$ and $z\in [0,\infty)$, set
\begin{equation}\label{genLCB}\mathscr{L}f(z):=z\mathrm{L}^{\Psi}f(z)-\frac{c}{2}z^2f'(z),
\end{equation}
with $\mathrm{L}^{\Psi}$ the generator of a spectrally positive Lévy process with Laplace exponent $\Psi$, viz.
\[\mathrm{L}^{\Psi}f(z):=\int_{0}^{\infty}\left(f(z+y)-f(z)-yf'(z)\mathbbm{1}_{\{y<1\}}\right)\pi(\ddr y)-\gamma f'(z)+\frac{\sigma^2}{2}f''(z).\]
\begin{lemma}\label{lem:MPZ} 
The LCB process $Z$, stopped at its first explosion time, is the unique solution to the following martingale problem:\\ 

$\mathrm{(MP)}_Z$: for any $z>0$, and any $f\in C_c^2([0,\infty))$, the process
\[M:=(M_t,t\geq 0):=\left(f(Z_t)-\int_0^{t}\mathscr{L}f(Z_s)\ddr s,t\geq 0\right) \text{ is a } \mathbb{P}_z\text{-martingale}.\]
Moreover, $M$ is a true martingale for any $f\in \mathcal{D}_Z$ and is a local one for any $f\in C^2([0,\infty))$. \end{lemma}

\begin{proof}
Recall that the LCB process $Z$ is solution of the stochastic equation \eqref{SDELCB}. The claim that $M$ is a local martingale for any $f\in C^{2}([0,\infty))$ is a consequence of the application of It\^o's lemma on $f(Z_t)$, see e.g. Meyer \cite[Theorem 3, Chapter III and Theorem 21, Chapter IV]{zbMATH03583004}.  When $f\in \mathcal{D}_Z$, we have $\sup |\mathscr{L}f(z)|<\infty$. The local martingale $M$, being bounded on any interval $[0,t]$, is a true martingale. In particular, this is the case when $f\in C^2_c(0,\infty)$ and $Z$ satisfies the martingale problem $\mathrm{(MP)}_Z$. The fact that there is a unique solution of the latter under the condition of non-explosion $\mathcal{E}=\infty$ has been shown in \cite{MR3940763}.  \qed
\end{proof}
\noindent \textbf{Lamperti's time change}. Following an idea already present in \cite{MR2134113}, and further detailed in \cite[Section 4]{MR3940763}, the LCB process $Z$, starting from $z>0$, can be constructed as follows. Let $Y$ be a spectrally positive Lévy process with Laplace exponent $\Psi$, starting from $z>0$. Set 
\[R_t:=Y_t-\frac{c}{2}\int_{0}^{t}R_s\ddr s, \text{ for any }t\geq 0.\]
Set $\sigma_0:=\inf\{s>0: R_s\leq 0\}$ and define \[[0,\infty]\ni u\mapsto \theta_u:=\int_{0}^{u\wedge \sigma_0}\frac{\ddr s}{R_s}\in [0,\infty].\]
 The process $R$ is a generalized Ornstein-Uhlenbeck (GOU) process, see e.g. Sato's book \cite[Chapter 3.17]{MR3185174}. Notice that its generator is of the form $\mathscr{L}^{R}f(z):=\mathrm{L}^{\Psi}f(z)-\frac{c}{2}zf'(z)$, where $\mathrm{L}^{\Psi}$ is the generator of $Y$. The identity $z\mathscr{L}^{R}f(z)=\mathscr{L}f(z)$, together with Volkonski's theorem on time-changes of Markov processes, see e.g. Ethier-Kurtz's book \cite[Theorem 1.4 page 309]{ethier}, ensures that the process $(Z_t,t\geq 0)$, defined by 
\begin{align}\label{eq:timechange}
Z_t&=
\begin{cases}
R_{C_t} &  0\leq t<\theta_\infty,\\
0& t\geq \theta_\infty \text{ and }\sigma_0<\infty,\\
\infty& t\geq \theta_\infty   \text{ and } \sigma_0=\infty,
\end{cases}
\end{align}
with $C_t:=\inf\{u>0: \theta_u>t\}=\int_{0}^{t}Z_s\ddr s$, is a solution to $(\mathrm{MP})_Z$, and thus a weak solution to \eqref{SDELCB}, see e.g. Kurtz \cite[Theorem 2.3]{zbMATH05919793}.  In the case without competition, we recover Lamperti's time-change of CBs, see e.g. \cite[Theorem 12.2]{MR3155252}.

\subsubsection{Extinction and total progeny: proof of Proposition \ref{prop1}}\label{sec:preliminariesLCB}
The possible behaviors of the LCB process at the boundaries and in its long-term have been classified in \cite{MR3940763}, see the Introduction. 

Fix $x_0$ in $(0,\infty)$ and recall \begin{equation}\label{testnonexplosion}\mathcal{E}=\int_0^{x_0}\!\tfrac{1}{u}e^{\int_u^{x_{0}}\frac{2\Psi(v)}{cv}\ddr v}\ddr u.\end{equation}
The condition $\mathcal{E}=\infty$ is known to be necessary and sufficient for the GOU process $R$ to be recurrent, see Shiga \cite{MR1061937}. The latter is furthermore positive recurrent if and only if $\int^{\infty}\log(y)\pi(\ddr y)<\infty$.

When $R$ is recurrent, the random variable $\theta_\infty$ is necessarily infinite on the event $\{\sigma_0=\infty\}$. This explains that the LCB process does not explode. On the other hand, the assumption $\Psi(\infty)=\infty$ entails that $R$ can hit any point in $\mathbb{R}$, so that $\sigma_0$ is finite with positive probability. We can now proceed to show Proposition \ref{prop1}. This follows the same argument as Bingham \cite{BINGHAM1976217} who dealt with the pure CB case in which $c=0$.
\smallskip

\noindent \textbf{Proof of Proposition \ref{prop1}}. Recall that the condition \asp \ encompasses $\mathcal{E}=\infty$ and $\Psi(\infty)=\infty$. It ensures therefore that $\{\sigma_0<\infty\}$ is almost sure. By the time-change relationship \eqref{eq:timechange}, $Z$ goes to $0$ a.s. on the event $\{\sigma_0<\infty\}$. Moreover $\sigma_0$ is nothing but the limit $C_\infty=\underset{t\rightarrow \infty}{\lim} C_t$. Therefore, up to a null set \[\{Z_t\underset{t\rightarrow \infty}{\longrightarrow} 0\}=\{\sigma_0<\infty\}=\Big\{C_\infty=\int_{0}^{\infty}\!Z_s\ddr s<\infty\Big\},\]
and we have $J=C_\infty$. One also sees that $\mathbb{E}_z(\sigma_0)=\mathbb{E}_z(J)<\infty$ if and only if $\int^{\infty}\log(y)\pi(\ddr y)<\infty$, since the log-moment is necessary and sufficient for $R$ to be positive recurrent, see \cite[Chapter 3.17]{MR3185174}. \qed


\subsection{Study of auxiliary diffusions}\label{sec:extincprogeny}
A striking feature of the LCBs is that they satisfy some duality relationships with certain \textit{one-dimensional diffusions}. They have been established in the articles \cite{MR3940763}, \cite{foucart2021local} and \cite{FOUCART2024104230}. 
We will briefly reexplain them in the setting of a non-exploding LCB which gets extinct a.s., see the forthcoming section \ref{sec:duality}. This will not cause any confusions. We introduce two positive diffusions $U$ and $V$. They will help us in our study of the LCB process in several ways.  The process $V$ lies in a certain \textit{biduality} relationship with the LCB process $Z$. 
\subsubsection{Laplace and Siegmund dualities}\label{sec:duality}
Define $e_x(z)=e_z(x)=e^{-xz}$ for any $x,z\in (0,\infty)$.
\begin{lemma}[Lemma 5.1 in \cite{MR3940763}]\label{lemmadualityLA}  One has
\begin{equation}\label{dualityLA}
\mathscr{L}e_x(z)=\mathscr{A}e_z(x),\end{equation}
with $\mathscr{A}$ the operator defined on $C^2([0,\infty))$ as follows: for any $g\in C^2([0,\infty))$ and $x\in (0,\infty)$,
\begin{equation}\label{generatorU0}\mathscr{A}g(x):=\frac{c}{2}xg''(x)-\Psi(x)g'(x).
\end{equation}
Moreover, under \asp, for all $z\in [0,\infty]$, $x\in (0,\infty)$  and $t\geq 0$, \[\mathbb{E}_z[e^{-xZ_t}]=\mathbb{E}_x[e^{-zU_t}],\]
with $U$ the diffusion, unique weak solution to 
\begin{equation}\label{sdeU}
\ddr U_t=\sqrt{cU_t}\ddr B_t-\Psi(U_t)\ddr t,\ U_0=x.
\end{equation}
\end{lemma}
\begin{lemma}[Lemma 5.2 in \cite{MR3940763}]\label{lem:Uexit} Assume $\mathcal{E}=\infty$. Then, the boundary $0$ of $U$ is an exit (i.e. $U$ hits $0$ with positive probability and stays there).
\end{lemma}
The process $U$, being stochastically monotone, admits a Siegmund dual process $V$.  We shall refer to it as the \textit{bidual} process of $Z$. 
\begin{lemma}[Proposition 2.1 in \cite{foucart2021local}]\label{lem:siegmund} Assume \asp . The Siegmund dual process $V$ of $U$, i.e. the process satisfying for all $x,y\in (0,\infty)$ and $t\geq 0$, \[\mathbb{P}_x(U_t<y)=\mathbb{P}_y(x<V_t),\]
is the diffusion, unique weak solution to
\begin{equation}\label{sdeV}
\ddr V_t=\sqrt{cV_t}\ddr B_t+\big(c/2+\Psi(V_t)\big)\ddr t, \quad V_0=y.
\end{equation}
The generator of $V$ acts on $C^2_c(0,\infty)$ as follows \begin{equation}\label{generatorG} \mathscr{G}f(x):=\frac{c}{2}xf''(x)+(c/2+\Psi(x))f'(x).
\end{equation}
\end{lemma}
The classification of the boundaries of $V$ has been done in \cite{foucart2021local}. The assumption $\mathcal{E}=\infty$ entails in particular that the boundary $0$ is an entrance boundary (i.e. is inaccessible and non-absorbing). The assumption $\Psi(\infty)=\infty$, together with $\mathcal{E}=\infty$, ensures that $\infty$ is an almost sure attracting boundary of $V$, that is to say $V_t\underset{t\rightarrow \infty}{\longrightarrow} \infty$ a.s.. Grey's condition $\int^{\infty}\frac{\ddr x}{\Psi(x)}<\infty$ (which entails $\Psi(\infty)=\infty$) is necessary and sufficient for $\infty$ to be an exit point of $V$, namely the process hits $\infty$ in finite time and stays there for good almost surely. 
\medskip

The link between $Z$ and $V$ is given by the following proposition. Its proof is in Section \ref{sec:proofpropbidual} of the Appendix.
\begin{proposition}\label{lem:joiningduals1}  
For all $t\geq 0$ and all $z,x\in (0,\infty)$,
\begin{equation}\label{joiningduals}\mathbb{E}_z(e^{-xZ_t})=\int_{0}^{\infty}ze^{-zy}\mathbb{P}_{y}(V_t>x)\ddr y.
\end{equation}
More generally, let $n\geq 2$ and $x_1,\cdots,x_n\geq 0$, $0\leq t_1<t_2<\cdots<t_n$ and $z\in (0,\infty)$, 
\begin{align}\label{bidualfinitedimfromz}
&\mathbb{E}_z\left[e^{-x_1Z_{t_1}}\cdots e^{-x_{n-1}Z_{t_{n-1}}}(1-e^{-x_nZ_{t_n}})\right]\nonumber \\
&\qquad =\int_{\mathbb{R}_+}ze^{-zy}\mathbb{P}_{y}\big(V_{t_1}\geq x_1,\cdots, V_{t_{n-1}}\geq x_{n-1},V_{t_n}\leq x_n\big)\ddr y.
\end{align}
\end{proposition} 
\begin{remark}
The identity  \eqref{joiningduals} is a simple combination of the Laplace duality and Siegmund duality (in this order). This observation allows us to relate an excessive function of $V$ to one of $Z$. Similar techniques can be found in \cite{foucart2021local}. The identity for the finite-dimensional laws will not be used later in the proofs. 
\end{remark}
The relationship \eqref{joiningduals}  actually exchanges the  role of the boundaries $\infty$ of $V$ and $0$ of $Z$, see \cite[Table 6]{foucart2021local}.
\begin{table}[h!]
\begin{center}
\begin{tabular}{|c|c|c|}
\hline
Integral condition & Boundary of $V$ &  Boundary  of $Z$ \\
\hline
$\mathcal{E}=\infty$ & $0$ entrance &  $\infty$  entrance  \\
\hline
$\Psi(\infty)=\infty \text{ and }\int^{\infty}\frac{\ddr x}{\Psi(x)}=\infty$ & $\infty$ attracting almost surely &  $0$  attracting almost surely\\
\hline
$\int^{\infty}\frac{\ddr x}{\Psi(x)}<\infty (\Longrightarrow \Psi(\infty)=\infty)$ & $\infty$ exit &  $0$  exit\\
\hline
\end{tabular}
\vspace*{3mm}
\caption{Boundaries of $V,Z$.}
\label{correspondanceVZ}
\end{center}
\end{table}

We emphasize that, in all cases (whether or not Grey's condition is met), the boundary at $\infty$ for $V$ and the boundary at $0$ for $Z$ are absorbing. This means that if the process starts at the boundary, it remains there indefinitely.

\subsubsection{Study of the scale function and speed measure of $V$}
The speed measure of $V$ is given by 
\begin{align}\label{speedmeasureV}
m(\ddr y)&=m(y)\ddr y
:=e^{\int_{x_0}^{y}\frac{2\Psi(u)}{cu}\ddr u} \ddr y.
\end{align}
Since under the assumption \asp, $\infty$ is an attracting boundary for $V$, any scale function of $V$ is finite at $\infty$. Denote by $S$  the scale function of $V$ such that $S(\infty)=0$. Recall that this is the unique positive solution to $\mathscr{G}S=0$, vanishing at $\infty$, up to a multiplicative constant. 

One has \begin{equation}\label{scalefunctionV}
S(x):=\int_{x}^{\infty}\frac{1}{y}e^{-\int_{x_0}^{y}\frac{2\Psi(u)}{cu}\ddr u}\ddr y=\int_{x}^{\infty}\frac{1}{ym(y)}\ddr y,\quad \forall x>0.
\end{equation}
We stress that $\mathcal{E}=\infty$ is equivalent to $S(0)=\infty$. Recall the measure $s$ defined in \eqref{eq:s}. We have $s(\ddr x)=-S'(x)\ddr x=\frac{1}{x}e^{-\int_{x_0}^{x}\frac{2\Psi(u)}{cu}\ddr u}\ddr x$. Moreover, the tail of $s$ is $S$, i.e.  $\int_{x}^{\infty}s(\ddr y)=S(x)$, $\forall x\geq 0$.
\medskip

We gather in the next technical lemma some properties of the scale function and speed measure that will be useful for our study.

\begin{lemma}\label{lem:estimates} Assume \asp . The following holds: 
\begin{enumerate}[label=\roman*)]
\item \begin{align}
&\int_{0+} S(x)m(x)\ddr x<\infty,\label{eq1}
\end{align}
\item \begin{equation}\label{vartheta}
x(-S)'(x)\underset{x\rightarrow 0}{\longrightarrow} \ell :=\begin{cases}
e^{\int_0^{x_0}\frac{2\Psi(u)}{cu}\ddr u}\in (0,\infty) &\text{ if } \int^{\infty}\log y\,  \pi(\ddr y)<\infty,\\
0\, &\text{ if } \int^{\infty}\log y\,  \pi(\ddr y)=\infty,\end{cases}
\end{equation}
\item 
\begin{equation}\label{eq2}
\int_0^{\infty}x(-S)'(x)\ddr x<\infty. 
\end{equation}
\end{enumerate}
\end{lemma}
\begin{proof}
\begin{enumerate}[label=\roman*)]
\item Notice that $\Psi(x)\longrightarrow \Psi(0)=0$, as $x\rightarrow 0$, by continuity of $\Psi$. Hence, by Karamata's representation theorem, see \cite[Theorem 1.3.1, page 13]{regularvariation}, both functions $m:y\mapsto e^{-\int_y^{x_0}\frac{2}{c}\frac{\Psi(u)}{u}\ddr u}$ and $1/m$ are slowly varying at $0$. By \cite[Proposition 1.5.9a, page 26]{regularvariation}, this is also the case for $S:x\mapsto \int_{x}^{x_0}\frac{\ddr y}{y}\frac{1}{m(y)}$. The slow variation property of $S$ ensures that for any $\eta>0$,  $x^{\eta}S(x)\underset{x\rightarrow 0}{\rightarrow} 0$, see \cite[Proposition 1.3.6, page 16]{regularvariation}. Let $\delta\in (0,1)$. By choosing $\eta>0$ such that $\eta+\delta<1$, we see that $\int_{0+}x^{-\delta}S(x)\ddr x=\int_{0+}x^{-(\delta+\eta)}x^{\eta}S(x)\ddr x<\infty$. Similarly, the slow variation of $m$ ensures that $x^{\delta}m(x)\underset{x\rightarrow 0}{\longrightarrow} 0$ and we get \begin{center}$\int_{0+} S(x)m(x)\ddr x=\int_{0+} S(x)x^{-\delta}x^{\delta}m(x)\ddr x<\infty$,\end{center} which establishes \eqref{eq1}.
Notice that the assumption $\mathcal{E}=\infty$ was not used here, only the fact that $\Psi(0)=0$.
\item We now establish \eqref{vartheta}. One has $-S'(x)=\frac{1}{x}\exp\left(-\int_{x_0}^{x}\frac{2\Psi(u)}{cu}\ddr u\right)$, thus \begin{equation}\label{eq:xS'x}
x(-S)'(x)=\exp\left(\int_{x}^{x_0}\frac{2\Psi(u)}{cu}\ddr u\right)=1/m(x).
\end{equation}
Since, as $x$ goes to $0$, $\int_x^{x_0}\frac{\Psi(u)}{u}\ddr u$ either tends to $-\infty$ or has a finite limit, the function $x\mapsto x(-S)'(x)$ converges  as $x$ goes to $0$ towards the finite nonnegative limit $\exp\left(\int_{0}^{x_0}\frac{2\Psi(u)}{cu}\ddr u\right)$, the latter is $0$ if and only if $\int_{0}^{x_0}\frac{\Psi(u)}{u}\ddr u=-\infty$ which is equivalent to $\int^{\infty} \log y\, \pi(\ddr y)=\infty$. 
\item The fact that $x(-S)'(x)$ admits a finite limit at $0$ ensures that  $\int_{0+} x(-S)'(x)\ddr x<\infty$. We now check that $\int^{\infty} x(-S)'(x)\ddr x<\infty$. Recall that by assumption $\Psi(\infty)=\infty$. Without loss of generality, assume that $x_0$ is large enough so that $\Psi(x_0)>0$. The Lévy-Khintchine form of $\Psi$, see \eqref{eq:branchingmechanism}, and the convexity it induces, ensures that $\Psi(u)/u$ is non-decreasing and positive on $[x_0,\infty)$. Hence $0<b:=\Psi(x_0)/x_0\leq \Psi(u)/u$, for all $u\geq x_0$, and by \eqref{eq:xS'x}, one has, for some constant $C>0$, the bound $x(-S)'(x)\leq Ce^{-\frac{2b}{c}x}$, which is integrable near $\infty$.
%
\qed
\end{enumerate}
\end{proof}

\subsubsection{The diffusion $V$ conditioned on not being attracted by $\infty$}
We now introduce a certain Doob's transform of the process $V$. We refer the reader to Borodin and Salminen's book \cite[Chapter II, Sections 31 and 32]{MR1912205}, see also Salminen \cite{zbMATH03888652} and Evans and Hening \cite{zbMATH07074435}, for a comprehensive account on Doob's $h$-transforms for diffusions. 

\medskip

Recall $S$, the scale function of $V$, see \eqref{scalefunctionV}. The process $(S(V_t),t\geq 0)$ is a positive $\mathbb{P}_x$-local martingale, see e.g. \cite[Proposition 3.4]{MR1725357}. Denote by $\mathbb{P}_x^{\downarrow}$ the Doob's transform of $(V,\mathbb{P}_x)$ based on the positive excessive function $S$ (notice that contrary to
the setting of Theorem \ref{doobtransform1}, here  $S(0)=\infty$,  $S(\infty)=0$, by Assumption \asp. Denote by $\sigma$  the lifetime of $(V,\mathbb{P}_x^{\downarrow})$. For all stopping time $T$ and $A\in \mathcal{F}^V_T$ (the filtration of $V$), we have 
\begin{equation}\label{doobVt}\mathbb{P}_x^{\downarrow}\big(A,T<\sigma\big)=\frac{1}{S(x)}\mathbb{E}_x\big(S(V_{T})\mathbbm{1}_A\big).
\end{equation}
\begin{lemmastar}\label{lem:lifetimeVdown} For all $x\in (0,\infty)$, $\sigma=T_0:=\inf\{t>0: V_t=0\}$,  $\mathbb{P}_x^{\downarrow}$-almost surely.
\end{lemmastar}
\begin{proof} This can be seen in a similar way as in the proof of Lemma \ref{lem:strictlocalmartingale}, see also Remark \ref{rem:Follmermeasure}. Define the sequence of stopping times $\sigma_n:=T^{-}_{1/n}\wedge n$ and $T^{-}_{1/n}=\inf\{t>0: V_t\leq 1/n\}$. Since under \asp, $V_t\underset{t\rightarrow \infty}{\longrightarrow} \infty$, $\mathbb{P}_x$-a.s. one has $\sigma_n\underset{n\rightarrow \infty}{\longrightarrow} \infty$, $\mathbb{P}_x$-a.s. Furthermore, the map $S$ is decreasing, thus $S(V_{t\wedge \sigma_n})\leq S(1/n)<\infty$ a.s. and $(\sigma_n)$ actually localizes $(S(V_t),t\geq 0)$. 
By the same arguments as in Lemma \ref{lem:strictlocalmartingale}, the lifetime of the process under $\mathbb{P}_x^{\downarrow}$ is thus either infinite or $\sigma=\underset{n\rightarrow \infty}{\lim} \sigma_n=T_0$. \qed
\end{proof}

The next result, which holds for general one-dimensional diffusions, see for instance \cite{zbMATH03888652} and Perkowski and Ruf \cite{zbMATH06098198}, characterizes the law of the process $(V,\mathbb{P}_x^{\downarrow})$.

\begin{propstar}\label{propVdown}  
Assume \asp . The process $(V,\mathbb{P}_x^{\downarrow})$
is a diffusion with speed measure and scale function given by
\begin{align}
m^{\downarrow}(\ddr y)&=S(y)^2 m(\ddr y), \label{speedVdown}\\
S^{\downarrow}(y)&=\frac{1}{S(y)}. \label{scaleVdown}
\end{align} 
\end{propstar}
\begin{proof} 
The speed measure and the scale function of the process $(V,\mathbb{P}^{\downarrow}_x)$ are identified by applying the results in 
\cite[Chapter II. Section 31]{MR1912205} (with in their notation $h=S$ and $\alpha\equiv 0$). The identity \eqref{speedVdown} is a direct application. For the scale function, one has
\[S^{\downarrow}(y)=\int_{0}^{y}\frac{-S'(x)}{S(x)^2}\ddr x=\frac{1}{S(y)}-\frac{1}{S(0)}.\]
By assumption $\mathcal{E}=\infty$, therefore $S(0)=\infty$ and one gets \eqref{scaleVdown}.\qed
\end{proof}
\begin{remark}
The process $(V,\mathbb{P}_y^{\downarrow})$ can be seen as the diffusion $V$ conditioned on not being attracted by $\infty$. More precisely, the process if forced to go below any arbitrarily small positive levels, see for instance \cite[Corollary 3.4]{zbMATH06098198}. We shall however not need the definition of this conditioning later on.
\end{remark}
The following lemma will be crucial in the study of the LCB process.
\begin{lemma}\label{lem:localmartingale} Assume \asp . For all $x\in (0,\infty)$, the process $(V, \mathbb{P}_x^{\downarrow})$ has almost surely a finite lifetime, namely it hits $0$, $\mathbb{P}_x^{\downarrow}$-a.s.. Moreover, $\mathbb{P}^{\downarrow}_x(T_0>t)=\mathbb{E}_x\left(\frac{S(V_t)}{S(x)}\right)$ for any $t\geq 0$, and the process $(S(V_t),t\geq 0)$, under $\mathbb{P}_x$, is a strict local martingale.
\end{lemma}
\begin{proof}
Since $S^{\downarrow}(\infty)=1/S(\infty)=1/0=\infty$ and $S^{\downarrow}(0)=1/S(0)=1/\infty=0$, we see by applying standard results on diffusions, see e.g.   Karlin and Taylor's book  \cite[Proposition 5.22, page 345]{zbMATH03736679}, that the process $(V,\mathbb{P}^{\downarrow})$ has its boundary $0$ almost surely attracting, i.e.
\[\mathbb{P}^{\downarrow}_v(V_t\underset{t\rightarrow \infty}{\longrightarrow} 0)=1.\]
We now apply Feller's tests, see e.g. \cite[Chapter 15, Section 6, page 231]{zbMATH03736679} to see whether the boundary $0$ is accessible. The process $(V,\mathbb{P}^{\downarrow})$ hits $0$ if and only if
\[I(0)=\int_0^{x_0}S^{\downarrow}(x)m^{\downarrow}(\ddr x)<\infty.\]
Using \eqref{scaleVdown} and \eqref{speedVdown}, it turns out that
$I(0)=\int_0^{x_0}\frac{1}{S(x)}S(x)^2 m(\ddr x)=\int_0^{x_0} S(x)m(\ddr x)$, which is finite by Lemma \ref{lem:estimates}.

By applying \eqref{doobVt} with $A=\Omega$ and $T=t$, and Lemma \ref{lem:lifetimeVdown}, one obtains $\mathbb{P}^{\downarrow}_x(T_0>t)=\mathbb{E}_x\left(\frac{S(V_t)}{S(x)}\right)$ for any $t\geq 0$. If $(S(V_t),t\geq 0)$ is a martingale, it has then constant expectation and we see that $T_0=\infty$, $\mathbb{P}^{\downarrow}_x$-almost surely, which contradicts the fact that $0$ is accessible for the process under $\mathbb{P}_x^{\downarrow}$. Thus $(S(V_t),t\geq 0)$ must be a strict local martingale.
\qed
\end{proof}
\begin{remark} The method used in the proof of Lemma \ref{lem:localmartingale} to show that the process $(S(V_t),t\geq 0)$ is a strict local martingale can be found in various forms in the literature, we refer e.g. to the articles by Pal and Protter \cite[Proposition 1]{zbMATH05763682}, see also Mijatovic and Urusov \cite[Theorem 2.1]{zbMATH06010474}.
\end{remark}

\section{Proof of Theorem \ref{thm1}}\label{sec:prooftheorem1}
Recall that the function $S$ is defined in \eqref{scalefunctionV}.
\begin{lemma}\label{lem:defh}
Assume \asp . The function $h$ given in \eqref{hintro} is well-defined on $[0,\infty)$. Moreover, for any $z\geq 0$, \begin{equation}\label{hwithS} h(z)=\int_{0}^{\infty}(1-e^{-xz})(-S)'(x)\ddr x=\int_0^{\infty}ze^{-zy}S(y)\ddr y.
\end{equation}
One has $h(0)=0$, $h(z)/z\leq h'(0)<\infty$. Under the assumption $\mathcal{E}=\infty$, $S(0)=\infty$ and $h(z)\underset{z\rightarrow \infty}{\longrightarrow} \infty$. When $\int^{\infty}\log u \pi(\ddr u)<\infty$, that is to say $\ell>0$, see \eqref{vartheta}, one has
\begin{equation}\label{equivlogcase}
h(z)\underset{z\rightarrow \infty}{\sim} \ell \log z.
\end{equation} 
Lastly, we have in both cases $\ell=0$ and $\ell>0$, 
\begin{equation}\label{hmoment}
\int^{\infty}h(y)\pi(\ddr y)<\infty.
\end{equation}
\end{lemma}
\begin{proof}
The first identity in \eqref{hwithS} holds by definition of $h$ in \eqref{hintro}. The function $h$ is well-defined by using the bound $1-e^{-xz}\leq xz$ and the integrability \eqref{eq2}. One has indeed for any $z\geq 0$,
\[h(z)\leq z\int_{0}^{\infty}x(-S)'(x)\ddr x<\infty.\]
For any $z>0$, by integration by parts.
\begin{align*}
\int_0^{\infty}ze^{-zy}S(y)\ddr y&=-\int_0^{\infty}(1-e^{-zy})S'(y)\ddr y+[(1-e^{-yz})S(y)]_{y=0}^{y=\infty}. \label{integrationbypartsinh}
\end{align*}
It remains to study the bracket terms. It vanishes at $y=\infty$, since $S(\infty)=0$. For $y=0$, using the bound $1-e^{-yz}\leq zy$, and cutting the integral form of $S$ into two parts, one from $x_1$ to $\infty$, and one from $y$ to $x_1$, one gets
\[yS(y)=yS(x_1)+y\int_y^{x_1}\frac{\ddr x}{x}e^{\int_x^{x_0}\frac{2\Psi(u)}{cu}\ddr u}\ddr x\leq o(1)+\int_0^{x_1}e^{\int_x^{x_0}\frac{2\Psi(u)}{cu}\ddr u}\ddr x.\]
The integral in the upper bound above matches with $\int_0^{x_1}xS'(x)\ddr x$ which is finite by \eqref{eq2}. The latter is arbitrarily small for $x_1$ close enough to $0$ and we have $\underset{y\rightarrow 0}{\lim}(1-e^{-yz})S(y)=0$. The fact that $h(0)=0$ is obvious and we have $h(\infty)=\infty$ since $S(0)=\infty$. Notice that $h$ is a Bernstein function, so in particular it is twice continuously differentiable on $(0,\infty)$ and one has $$h'(z)=\int_{0}^{\infty}xe^{-xz}(-S)'(x)\ddr x.$$
By \eqref{eq2}, $\int^{\infty}x(-S)'(x)\ddr x<\infty$, hence $h'(0)<\infty$. We check now \eqref{equivlogcase}. Assume $\int^{\infty}\log u \pi(\ddr u)<\infty$. Recall \eqref{vartheta} and $\underset{x\rightarrow 0}{\lim} x(-S)'(x)=\ell>0$. One has $\int_0^{x}y(-S)'(y)\ddr y \underset{x\rightarrow 0}{\sim} \ell x$ and by Karamata's Tauberian's theorem, see e.g. \cite{regularvariation}, $h'(z)\underset{z\rightarrow \infty}{\sim} \ell/z$, which together with the fact that $h(0)=0$ entails that
$$h(z)=\int_0^{z}h'(y)\ddr y \underset{z\rightarrow \infty }{\sim} \ell \log z.$$
It remains to check \eqref{hmoment}. Recall $\Psi$ in \eqref{eq:branchingmechanism}. One has for all $x\geq 0$, \begin{equation}\label{boundd}\int_{1}^{\infty}(1-e^{-xz})\pi(\ddr z)=-\Psi(x)+\int_{0}^{1}(e^{-xz}-1+xz)\pi(\ddr z)\leq -\Psi(x)+Cx^2,
\end{equation}
for some constant $C>0$. Recall the expression of $h$, $$h(z)=\int_0^\infty(1-e^{-xz})\frac{1}{x}e^{-\int_{x_0}^x\frac{2}{c}\frac{\Psi(u)}{u}\ddr u}\ddr x,$$ one gets by \eqref{boundd} and Fubini-Tonelli's theorem
\begin{equation}\label{eq:upperbound}\int_1^\infty h(z)\pi(\ddr z)\leq \int_{0}^{\infty}\left(-\frac{\Psi(x)}{x}+Cx\right)e^{-\int_{x_0}^{x}\frac{2}{c}\frac{\Psi(u)}{u}\ddr u}\ddr x.
\end{equation}

Recall that $\Psi(\infty)=\infty$ entails that $\frac{\Psi(u)}{u}\underset{u\rightarrow \infty}{\longrightarrow} \infty$. Hence for any $b>0$ fixed, one has $\frac{\Psi(u)}{u}\geq b$ for large enough $u$ and  $e^{-\int_{x_0}^{x}\frac{2}{c}\frac{\Psi(u)}{u}\ddr u}\leq C'e^{-bx}$ for some constant $C'>0$. This implies that
\[\int_1^{\infty}Cxe^{-\int_{x_0}^{x}\frac{2}{c}\frac{\Psi(u)}{u}\ddr u}\ddr x<\infty,\]
which entails the integrability near $\infty$ of the upper bound in \eqref{eq:upperbound}. On the other hand, since, $\ell=e^{\int_{0}^{x_0}\frac{2}{c}\frac{\Psi(u)}{u}\ddr u}\in [0,\infty)$, one has
$\int_0^1 xe^{\int_{x}^{x_0}\frac{2}{c}\frac{\Psi(u)}{u}\ddr u}\ddr x<\infty$ and
 \[\int_{0}^{1}\frac{-\Psi(x)}{x}e^{-\int_{x_0}^{x}\frac{2}{c}\frac{\Psi(u)}{u}\ddr u}\ddr x=\frac{c}{2}\left[e^{-\int_{x_0}^{x}\frac{2}{c}\frac{\Psi(u)}{u}\ddr u}\right]_{x=0}^{x=1}<\infty.\] \qed
\end{proof}
\begin{lemma}\label{lem:Lh} The function $h$ satisfies for all $z\geq 0$,
$\mathscr{L}h(z)=-\frac{c\ell}{2}z$, with $\ell$ defined in \eqref{vartheta}. It is harmonic, i.e. $\mathscr{L}h=0$ if and only if $\ell=0$, or equivalently $\int^{\infty} \log y\,\pi(\ddr y)=\infty$.
\end{lemma}
\begin{proof}
Let $z\in (0,\infty)$. Recall $e_x(z)=e_z(x)=e^{-xz}$ for all $x\in (0,\infty)$ and $\mathscr{L}e_x(z)=\mathscr{A}e_z(x)$ for all $x,z\in (0,\infty)$, see Lemma \ref{lemmadualityLA}. Note that $\mathscr{L}1=0=\mathscr{A}1$, therefore $\mathscr{L}(1-e_x)(z)=\mathscr{A}(1-e_z)(x)$ for any $x,z\in (0,\infty)$. For any $\theta>0$ and $z>0$,
\begin{align}
\mathscr{L}h(z)&=\int_{0}^{\infty}\mathscr{L}(1-e_x)(z)(-S)'(x)\ddr x=\int_0^{\infty}\mathscr{A}(1-e_z)(x)(-S)'(x)\ddr x \nonumber\\
&=\int_0^{\infty}\frac{c}{2}x(1-e_z)''(x)(-S)'(x)\ddr x-\int_0^{\infty}\Psi(x)(1-e_z)'(x)(-S)'(x))\ddr x \label{eqtoplug1}.
\end{align}
By integration by parts
\begin{align}
\int_0^{\infty}&\frac{c}{2}x(1-e_z)''(x)(-S)'(x)\ddr x \nonumber\\
&=\left[(1-e_z)'(x)\frac{c}{2}x (-S)'(x)\right]_{x=0}^{x=\infty}-\int_0^{\infty}(1-e_z)'(x)\frac{c}{2} \left((-S)'(x)+x(-S)''(x)\right)\ddr x \nonumber.
\end{align}
We study now the bracket terms. One has
\begin{align*}
\left[(1-e_z)'(x)\frac{c}{2}x (-S)'(x)\right]_{x=0}^{x=\infty}&=
-\underset{x\rightarrow \infty}{\lim} ze^{-zx}\frac{c}{2}x S'(x)+\underset{x\rightarrow 0}{\lim} ze^{-zx}\frac{c}{2}xS'(x)=-\frac{c\ell}{2}z,
\end{align*}
with $\ell\geq 0$ defined in \eqref{vartheta}.
Hence
\begin{equation}
\int_0^{\infty}\frac{c}{2}x(1-e_z)''(x)(-S)'(x)\ddr x =-\frac{c\ell}{2}z+\int_0^{\infty}(1-e_z)'(x)\frac{c}{2}\left(S'(x)+xS''(x)\right)\ddr x,\label{eqtoplug2}
\end{equation}
and by plugging \eqref{eqtoplug2} in \eqref{eqtoplug1} and recalling the generator $\mathscr{G}$ of $V$, see \eqref{generatorG}, we have
\begin{align*}
\mathscr{L}h(z)&=-\frac{c\ell}{2}z+\int_0^{\infty}(1-e_z)'(x)\left(\frac{c}{2} \big(S'(x)+xS''(x)\big)+\Psi(x) S'(x)\right)\ddr x\\
&=-\frac{c\ell}{2}z+\int_0^{\infty}ze^{-zx}\mathscr{G}S(x)\ddr x=-\frac{c\ell}{2}z.
\end{align*}
\end{proof}

\begin{lemma}\label{lem:hexcessive}
The process $\left(h(Z_t)+\int_0^{t}\frac{c\ell}{2}Z_s\ddr s,t\geq 0\right)$ is a local martingale.  Moreover, 
\begin{enumerate}[label=\roman*)]
\item $(h(Z_t),t\geq 0)$ is a strict local martingale if $\ell=0$, namely when $\int^{\infty} \log y\,\pi(\ddr y)=\infty$,
\item $(h(Z_t),t\geq 0)$ is a strict supermartingale when $\ell>0$, namely when $\int^{\infty} \log y\,\pi(\ddr y)<\infty$.
\end{enumerate}
\end{lemma}
\begin{remark}
Note that in both cases $\ell=0$ and $\ell>0$, the process $(h(Z_t),t\geq 0)$ is a positive supermartingale. 
\end{remark}
\begin{proof}
Recall that the function $h$ is $C^2$. By applying Itô's lemma, to the semimartingale $Z$ solution to \eqref{SDELCB}, we see that 
$$\left(h(Z_t)-\int_0^{t}\mathscr{L}h(Z_s)\ddr s,t\geq 0\right)$$ is a local martingale. By Lemma \ref{lem:Lh}, $\mathscr{L}h(z)=-\frac{c\ell}{2}z$. This provides the targeted local martingale. 

It remains to verify that when $\ell=0$, it is a strict local martingale. Let $z\geq 0$ and $t\geq 0$, by Proposition~\ref{lem:joiningduals1}, see \eqref{joiningduals}, and by applying Fubini-Tonelli's theorem, one gets
\begin{align*}
\mathbb{E}_z[h(Z_t)]&=-\int_0^{\infty}\mathbb{E}_z(1-e^{-xZ_t})S'(x)\ddr x=-\int_0^{\infty}\int_{0}^{\infty}ye^{-zy}\mathbb{P}_y(V_t\leq x)S'(x)\ddr x\ddr y\\
&=\int_0^{\infty}ye^{-zy}\mathbb{E}_y[S(V_t)]\ddr y.
\end{align*}
By Lemma \ref{lem:localmartingale}, $\mathbb{E}_y\big(\frac{S(V_t)}{S(y)}\big)=\mathbb{P}_y^{\downarrow}(t<T_0)\in (0,1)$ for any $t\geq 0$ and $y\geq 0$. Therefore by the identity \eqref{hwithS}, we get for any $t\geq 0$ and $z\geq 0$
$$h(z)-\mathbb{E}_z[h(Z_t)]=\int_{0}^{\infty}ye^{-zy}\big(S(y)-\mathbb{E}_y[S(V_t)]\big)\ddr y>0.$$  The process $(h(Z_t),t\geq 0)$ is a supermartingale and not a martingale. In particular in the case $\ell=0$, we deduce that $(h(Z_t),t\geq 0)$ is a strict local martingale. 
\qed
\end{proof}

\noindent \textbf{Proof of Theorem \ref{thm1}}. Lemma \ref{lem:defh} yields the first statement. The second is given by Lemma~\ref{lem:hexcessive}. \qed
\section{Proofs of Proposition \ref{prop:infimum} and Theorem \ref{mainthm}}\label{sec:prooftheorem2}
For all $z\in (0,\infty)$, define the measure $\mathbb{P}_z^{\uparrow}$ as the law of the Markov process with semigroup given by
\begin{equation}\label{defhtransform}P_t^{\uparrow}f(z):=\mathbb{E}_z^{\uparrow}\big(f(Z^{\uparrow}_t),t<\zeta\big):=\frac{1}{h(z)}\mathbb{E}_z[h(Z_t)f(Z_t)].\end{equation}
The process $(Z, \mathbb{P}_z^{\uparrow})$ is the $h$-transform of $Z$, see Section \ref{sec:doobtransform}. Recall that $\zeta$ denotes its lifetime. We choose $\infty$ as the cemetery state, so that $\zeta=\inf\{t>0:Z_t=\infty\}$, $\mathbb{P}_z^{\uparrow}$-a.s.. This convention is harmless, we shall indeed see that, in our context, $\infty$ is always absorbing\footnote{Recall however that when \asp \ holds, $\infty$ is an entrance boundary point for the LCB process} for $(Z,\mathbb{P}^{\uparrow}_z)$, meaning that when $f(\infty)=0$, $\underset{z\rightarrow \infty}{\lim} P^{\uparrow}_tf(z)=0$.

Recall that under $\mathbb{P}_z^{\uparrow}$, we extend the process $Z$ on $[\zeta,\infty)$ by setting $Z_{\zeta+t}=\infty$ for any $t\geq 0$, $\mathbb{P}_z^{\uparrow}$-a.s. In particular, one has for any continuous function $f$ vanishing at $\infty$, all $t\geq 0$ and $z\in (0,\infty)$, \[P_t^{\uparrow}f(z)=\mathbb{E}_z^{\uparrow}\big(f(Z_t),t<\zeta\big)
=\mathbb{E}_z^{\uparrow}\big(f(Z_t)\big).\] 
The auxiliary process $V^{\downarrow}$ and the identity \eqref{joiningduals} allow us to get a representation of the semigroup of $Z^{\uparrow}$ on the exponential functions. \begin{lemma}[Representation of the semigroup $P^{\uparrow}_t$]\label{semigroupofPuparrow} Assume \asp . For all $x,z\in (0,\infty)$ and $t\geq 0$, we have
\begin{align}
\mathbb{E}_z^{\uparrow}[e^{-xZ_t}]&=\frac{1}{h(z)}\int_{0}^{\infty}ze^{-zy}\mathbb{E}_y\left[\mathbbm{1}_{\{x\leq V_t\}}S(V_t-x)\right]\ddr y \label{semigroupPtuparrow1}\\
&=\frac{1}{h(z)}\int_{0}^{\infty}ze^{-zy}S(y)\mathbb{E}^{\downarrow}_y\left[\mathbbm{1}_{\{x\leq V_t\}}\frac{S(V_t-x)}{S(V_t)},t<T_0\right]\ddr y, \label{semigroupPtuparrow2}
\end{align}
where we recall $T_0:=\inf\{t>0: V_t=0\}$.
\end{lemma}
\begin{proof}
Recall $S(\infty)=0$, the expression of $h$ in \eqref{lem:defh} and the relationship \eqref{joiningduals} between the semigroups of $Z$ and $V$. For any $z,x>0$ and $t\geq 0$,
\begin{align*}
h(z)\mathbb{E}_z^{\uparrow}[e^{-xZ_t}]&=\int_{0}^{\infty}\mathbb{E}_z[e^{-xZ_t}-e^{-(x+v)Z_t}](-S'(v))\ddr v\\
&=\int_{0}^{\infty}\int_{0}^{\infty}ze^{-zy}\big(\mathbb{P}_y(V_t\geq x)-\mathbb{P}_y(V_t\geq x+v)\big)(-S'(v))\ddr v \ddr y\\
&=\int_{0}^{\infty}\int_{0}^{\infty}ze^{-zy}\mathbb{P}_y(x\leq V_t<x+v)\big(-S'(v)\big)\ddr v \ddr y\\
&=\int_{0}^{\infty}ze^{-zy}\mathbb{E}_y\big(\mathbbm{1}_{\{x\leq V_t\}}S(V_t-x)\big)\ddr y\\
&=\int_{0}^{\infty}ze^{-zy}S(y)\mathbb{E}_y\left(\mathbbm{1}_{\{x\leq V_t\}}\frac{S(V_t-x)}{S(V_t)}\frac{S(V_t)}{S(y)},t<T_\infty\right)\ddr y\\
&=\int_{0}^{\infty}ze^{-zy}S(y)\mathbb{E}^{\downarrow}_y\left(\mathbbm{1}_{\{x\leq V_t\}}\frac{S(V_t-x)}{S(V_t)},t<T_0\right)\ddr y.
\end{align*}
The fourth line above is obtained by applying Fubini-Tonelli's theorem. It provides \eqref{semigroupPtuparrow1}. The penultimate equality holds true since $(V,\mathbb{P}_y)$ has its boundary $\infty$ absorbing and $S(\infty)=0$. The last line, which gives \eqref{semigroupPtuparrow2}, is deduced from \eqref{doobVt}, where we recall that $\sigma=T_0$, $\mathbb{P}^{\downarrow}_y$-a.s., see Lemma \ref{lem:lifetimeVdown}, and from the fact that $(V,\mathbb{P}_y^{\downarrow})$ cannot hit $\infty$, i.e. $\mathbb{P}_y^{\downarrow}(t<T_\infty)=1$ for all $t\geq 0$. \qed
\end{proof}
\begin{lemma}\label{lem:finitelifetime}
For any $z>0$, recall $\zeta=\inf\{t>0:Z_{t}=\infty\}$, the lifetime of the process under $\mathbb{P}_z^{\uparrow}$. 
For all  $t>0$,
\[\mathbb{P}^{\uparrow}_z(\zeta>t)=\frac{1}{h(z)}\int_{0}^{\infty}ze^{-zy}S(y)\mathbb{P}_y^{\downarrow}(T_0>t)\ddr y\underset{t\rightarrow \infty}{\longrightarrow} 0,\]
and thus $(Z,\mathbb{P}_z^{\uparrow})$ has a finite lifetime a.s.. 
\end{lemma}
\begin{proof}
By letting $x$ go to $0$ in the identity \eqref{semigroupPtuparrow2}; and by  continuity under expectation, we see that
\[\mathbb{P}^{\uparrow}_z(\zeta>t)=\underset{x\rightarrow 0}{\lim}\mathbb{E}^{\uparrow}_z(e^{-xZ_t})=\frac{1}{h(z)}\int_{0}^{\infty}ze^{-zy}S(y)\mathbb{P}^{\downarrow}_y(T_0>t)\ddr y.\]
By Lemma \ref{lem:localmartingale}, for all $y\in (0,\infty)$, $\mathbb{P}_y^{\downarrow}(T_0=\infty)=0$. Therefore, by Lebesgue's theorem, we see that $\mathbb{P}^{\uparrow}_z(\zeta=\infty)=0$. \qed
\end{proof}

\subsection{Proof of Proposition \ref{prop:infimum}}\label{sec:corollary}
We study the infimum of the process $(Z,\mathbb{P}_z^{\uparrow})$ for $z>0$.
Set $g(z):=1/h(z)$ for any $z\in (0,\infty)$ and note that $g(0)=\underset{z\rightarrow 0}{\lim}g(z)=\infty$ and $g(\infty)=\underset{z\rightarrow \infty}{\lim}g(z)=0$. By definition of $\mathbb{P}^{\uparrow}_z$, see \eqref{defhtransform}, $$\mathbb{E}^{\uparrow}_z\big(g(Z_t)\big)=g(z)\mathbb{E}_z\big(h(Z_t)g(Z_t)\big)=g(z).$$
Hence the function is invariant for the semigroup $(P_t)$ and  $(g(Z_t),t\geq 0)$ is a martingale under $\mathbb{P}_z^{\uparrow}$. Denote by $\zeta^{-}_{a}$ the first passage time below level $a$.  We get by the optional stopping theorem applied to $t\wedge \zeta^{-}_{a}\wedge \zeta$ with $a\leq z$:
\[\mathbb{E}^{\uparrow}_z\big(g(Z_{t\wedge \zeta^{-}_{a}\wedge \zeta})\big)=g(z).\]
By the absence of negative jumps in $Z$ under $\mathbb{P}^{\uparrow}_z$, see Remark \ref{rem:nonegativejumps}-i), we have $Z_{\zeta_a^-}=a$ a.s. on $\{\zeta_a^-<\zeta\}$. Recall that by Lemma \ref{lem:finitelifetime}, $\zeta$ is finite almost surely. By letting $t$ go to $\infty$, we get 
\begin{center} 
$\mathbb{E}^{\uparrow}_z\big(g(Z_{\zeta_a^{-}\wedge \zeta})\big)=
g(a)\mathbb{P}^{\uparrow}_z(\zeta_a^{-}<\zeta)+\mathbb{E}^{\uparrow}_{z}\big(g(Z_{\zeta})(1-\mathbbm{1}_{\{\zeta_a^{-}<\zeta\}})\big)= g(z)$.\end{center}
Since $g(Z_\zeta)=g(\infty)=0$ a.s. and $\mathbb{P}_z^{\uparrow}\big(\zeta_a^{-}<\zeta\big)=\mathbb{P}_z^{\uparrow}\big(\underset{0\leq t<\zeta}{\inf} Z_t\leq a\big)$,  the proof is achieved. \qed

\subsection{Proof of Theorem \ref{mainthm}}
\begin{lemma}\label{lem:Feller}
The process $(Z,\mathbb{P}^{\uparrow}_z)$ is Feller in the sense that its semigroup maps $C_0$, the space of continuous functions defined on $(0,\infty)$, vanishing at $\infty$, into itself.
Moreover,
if $\int^{\infty}\log y\, \pi(\ddr y)<\infty$ then $Z^{\uparrow}_{\zeta-}<\infty$ with positive probability and 
if $\int^{\infty}\log y\, \pi(\ddr y)=\infty$ then $Z^{\uparrow}_{\zeta-}=\infty$ a.s. 
\end{lemma}
\begin{proof}
We first show the Feller property. Let $e_x(z):=e^{-xz}$ for any $x,z\geq 0$. The continuity of the map $P^{\uparrow}_te_x(\cdot)$ on $(0,\infty)$ is a direct application of the representation of the semigroup $P^{\uparrow}_t$ given in Lemma \ref{semigroupofPuparrow}, \eqref{semigroupPtuparrow1}. For the limit at $\infty$, note that by definition $h(z)P^{\uparrow}_te_x(z)=\mathbb{E}_z\big(e_x(Z_t)h(Z_t)\big)$. Recall that $h(z)\leq h'(0)z$, hence $z\mapsto h(z)e_x(z)$ is continuous vanishing at $\infty$. 
Moreover, the boundary point $\infty$ is an entrance for the LCB process $Z$, see Table \ref{correspondanceVZ}, therefore 
\[\underset{z\rightarrow \infty}{\lim}\, h(z)P^{\uparrow}_te_x(z)=\mathbb{E}_\infty[h(Z_t)e_x(Z_t)]<\infty.\] 
Finally since $h(z)\underset{z\rightarrow \infty}{\longrightarrow} \infty$, we have that $z\mapsto P^{\uparrow}_te_x(z)$ is vanishing at $\infty$. 
By the Stone-Weierstrass theorem, we see that $P_t^{\uparrow}C_0\subset C_0$ for any $t>0$.  Lastly, since the process is Feller, it admits a version with c\`adl\`ag sample paths, see e.g. \cite[Theorem 6, Chapter 2, page 54]{{zbMATH02208909}}, hence with a left limit at the lifetime $\zeta$. Lemma \ref{lem:hexcessive} entails then that when $\ell>0$, $(h(Z_t),t\geq 0)$ is a strict supermartingale (it is not a local martingale), Lemma \ref{lem:strictlocalmartingale} in Section \ref{sec:doobtransform} ensures then that the $h$-transformed process is killed with positive probability. At the contrary, when $\ell=0$, the process explodes continuously a.s. since  $(h(Z_t),t\geq 0)$ is a strict local martingale. \qed
\end{proof}
%
We now relate the excessive function $h$ to the conditioning along the total progeny. Recall $J=\int_0^\infty Z_s\ddr s$ and that under \asp, $J<\infty$, $\mathbb{P}_z$-almost surely for all $z\geq 0$. 
\begin{lemma}\label{lem:convergenceL1} Assume \asp . Let $\theta>0$. Set $$f_\theta(z):=\int_{0}^{\infty}x^{2\theta/c}e^{-xz}\frac{1}{x}e^{-\int_{x_0}^{x}\frac{2\Psi(u)}{cu}\ddr u}\ddr x=\int_{0}^{\infty}x^{2\theta/c}e^{-xz}(-S)'(x)\ddr x.$$
The function $f_\theta$ is well-defined decreasing and bounded (so that $f_\theta(0)<\infty$). Moreover 
\begin{equation}\label{totalprogeny}
\mathbb{E}_z[e^{-\theta J}]=\frac{f_\theta(z)}{f_\theta(0)},
\end{equation}
and for any $(\mathcal{F}_t)$-stopping time $T\geq 0$ and $z\geq 0$,
\begin{equation}\label{cvfthetatohZ_t}
\mathbb{E}_z\big(f_\theta(0)-f_\theta(Z_T)\big)\underset{\theta \rightarrow 0}{\longrightarrow} \mathbb{E}_z\big(h(Z_T)\big).
\end{equation}
\end{lemma}
\begin{proof} 
The functions $f_\theta,\theta>0$ have been studied in \cite{MR3940763}, see therein Lemma 4.4 and Equation (4.5). We recheck briefly that they are well-defined. Recall \eqref{eq:xS'x}, \[x(-S'(x))=e^{-\int_{x_0}^{x}\frac{2}{c}\frac{\Psi(u)}{u}\ddr u}.\] 
As seen in Lemma \ref{lem:estimates}-ii) , the latter is bounded when $x$ is close to $0$, and decays towards $0$ when $x$ goes to $\infty$ faster than $x^{-2b/c}$ for any $b>0$. By choosing $b>\theta$, we have that
$$f_\theta(z)\leq f_\theta(0)=\int_{0}^{\infty}x^{\frac{2\theta}{c}-1}x(-S'(x))\ddr x<\infty.$$
The identity \eqref{totalprogeny} can be found in \cite[Lemma 4.4]{MR3940763}. We now establish
\eqref{cvfthetatohZ_t}. By definition of $f_\theta$ and Fubini-Tonelli,
\[\mathbb{E}_z\big(f_\theta(0)-f_\theta(Z_T)\big)=\mathbb{E}_z\left[\int_0^\infty x^{2\theta/c}\left(1-e^{-xZ_T}\right)(-S'(x))\ddr x\right]=\int_0^\infty x^{2\theta/c}\mathbb{E}_z\left(1-e^{-xZ_T}\right)(-S'(x))\ddr x.\]
Plainly, the last integrand goes to $\mathbb{E}_z\left(1-e^{-xZ_T}\right)(-S'(x))$ as $\theta$ goes to $0$. One has furthermore the following bound, for any fixed $\theta_0>0$ and $\theta<\theta_0$,
\[x^{2\theta/c}\mathbb{E}_z\left(1-e^{-xZ_T}\right)(-S'(x))\leq \left(1+x^{2\theta_0/c}\right)\mathbb{E}_z\left(1-e^{-xZ_T}\right)(-S'(x)).\]
The integral with respect to $\ddr x$ of the upper bound equals $\mathbb{E}_z\big(h(Z_T)+f_{\theta_0}(Z_T)\big)$ which is finite since $f_{\theta_0}$ is bounded and $(h(Z_t),t\geq 0)$ is a positive supermartingale. We get by applying Lebesgue's theorem,

\[\int_0^\infty x^{2\theta/c}\mathbb{E}_z\left(1-e^{-xZ_T}\right)(-S'(x))\ddr x \underset{\theta \rightarrow 0}{\longrightarrow}\int_0^\infty \mathbb{E}_z\left(1-e^{-xZ_T}\right)(-S'(x))\ddr x=\mathbb{E}_z(h(Z_T)),\]
which yields \eqref{cvfthetatohZ_t}. \qed
\end{proof}
\begin{lemma}\label{lem:conditioningalongprogeny} Assume \asp . Let $\mathbbm{e}$ be a standard exponential random variable, independent of $Z$. One has, for any $(\mathcal{F}_t)$-stopping time $T$,  if $\Lambda\in \mathcal{F}_{T}$,
\[\mathbb{P}^{\uparrow}_z(\Lambda)=\underset{\theta \rightarrow 0}{\lim}\,\mathbb{P}_z(\Lambda, T\leq \mathbbm{e}/\theta|J\geq \mathbbm{e}/\theta).
\]
\end{lemma}
\begin{proof}
Set $\mathbbm{e}_\theta:=\mathbbm{e}/\theta$ for any $\theta>0$. Notice that $\mathbbm{e}_\theta$ is an exponential random variable with parameter $\theta$, to wit with mean $1/\theta$. For any stopping time $T$, set $J_T:=\int_{0}^{T}Z_t\ddr t$, one has by the strong Markov property, for any $A\in \mathcal{F}_T$
\begin{align}\label{calculation1}
\mathbb{P}_z(A,T<\mathbbm{e}_\theta| \mathbbm{e}_\theta<J)&=\frac{\mathbb{P}_z(A,T<\mathbbm{e}_\theta,\mathbbm{e}_\theta<J)}{\mathbb{P}_z(\mathbbm{e}_\theta<J)}  \nonumber \\
&=\frac{\mathbb{E}_z\big(\mathbbm{1}_{\{A,T<\mathbbm{e}_\theta\}}\mathbb{E}[\mathbbm{1}_{\{\mathbbm{e}_\theta<J\}}|\mathcal{F}_T]\big)}{\mathbb{P}_z(\mathbbm{e}_\theta<J)} \nonumber \\
&=\frac{\mathbb{E}_z\big(\mathbbm{1}_{\{A,T<\mathbbm{e}_\theta\}}\mathbb{E}[\mathbbm{1}_{\{\mathbbm{e}_\theta<J\circ \theta_T+J_T\}}|\mathcal{F}_T]\big)}{\mathbb{P}_z(\mathbbm{e}_\theta<J)}  \nonumber \\
&=\frac{\mathbb{E}_z\big(\mathbbm{1}_{\{A, T<\mathbbm{e}_\theta \,J_T<\mathbbm{e}_\theta\}}\mathbb{E}_{Z_T}[\mathbbm{1}_{\{\mathbbm{e}_\theta-J_T<J\}}]\big)}{\mathbb{P}_z(\mathbbm{e}_\theta<J)}.
\end{align} 
Conditionally on $T$ and on $\mathbbm{e}_\theta>J_T$,  by the lack of memory of the exponential distribution the random variable $\mathbbm{e}'_\theta:=\mathbbm{e}_\theta-J_T$ is also exponentially distributed with parameter $\theta$, therefore
\begin{align}\label{calculation2}
\frac{\mathbb{E}_z\big(\mathbbm{1}_{\{A,J_T<\mathbbm{e}_\theta, T<\mathbbm{e}_\theta\}}\mathbb{E}_{Z_T}[\mathbbm{1}_{\{\mathbbm{e}_\theta-J_T<J\}}]\big)}{\mathbb{P}_z(\mathbbm{e}_\theta<J)}&=\frac{\mathbb{E}_z\big(\mathbbm{1}_{\{A,T<\mathbbm{e}_\theta\}}e^{-\theta J_T}\mathbb{E}_{Z_T}[\mathbbm{1}_{\{\mathbbm{e}_\theta-J_T<J\}}|\mathbbm{e}_\theta>J_T]\big)}{\mathbb{P}_z(\mathbbm{e}_\theta<J)}\nonumber \\
&=\frac{\mathbb{E}_z\big(\mathbbm{1}_{\{A,T<\mathbbm{e}_\theta\}}e^{-\theta J_T}\mathbb{E}_{Z_T}[\mathbbm{1}_{\{\mathbbm{e}'_\theta<J\}}]\big)}{\mathbb{P}_z(\mathbbm{e}_\theta<J)} \nonumber \\
&=\mathbb{E}_z\Big(\mathbbm{1}_{\{A,T<\mathbbm{e}_\theta\}}e^{-\theta J_T}\frac{\mathbb{E}_{Z_T}(1-e^{-\theta J})}{\mathbb{E}_z(1-e^{-\theta J})}\Big).
\end{align}
Recall from Lemma \ref{lem:convergenceL1} that
$\mathbb{E}_z[e^{-\theta J}]:=\frac{f_\theta(z)}{f_\theta(0)}$.
By \eqref{cvfthetatohZ_t}, we have in $L^1(\mathbb{P}_z)$, 
\[\frac{\mathbb{E}_{Z_{T}}(1-e^{-\theta J})}{\mathbb{E}_z(1-e^{-\theta J})}=\frac{f_\theta(0)-f_\theta(Z_T)}{f_\theta(0)-f_\theta(z)}\underset{\theta \rightarrow 0}{\longrightarrow} \frac{h(Z_{T})}{h(z)}.\] 
Moreover, $e^{-\theta J_{T}} \underset{\theta \rightarrow 0}{\longrightarrow} 1$ since $J_{T}\leq J<\infty$ $\mathbb{P}_z$-a.s. Therefore, for any $\Lambda\in \mathcal{F}_{T}$,
\[\underset{\theta\downarrow 0}{\lim}\,\mathbb{P}_z(\Lambda, T<\mathbbm{e}_\theta| \mathbbm{e}_\theta<J)=\mathbb{E}_z\left(\frac{h(Z_{T})}{h(z)}\mathbbm{1}_{\Lambda}\right)=\mathbb{P}_z^{\uparrow}(\Lambda,T<\zeta).\]
\qed
\end{proof}

\noindent \textbf{Proof of Theorem \ref{mainthm}}. The first statement is given by Lemma \ref{lem:finitelifetime} and Lemma \ref{lem:Feller}. The second one is provided by Lemma \ref{lem:conditioningalongprogeny}. \qed
\section{Proof of Theorem \ref{thm:SDEZup}}\label{sec:prooftheorem34}
We now investigate the dynamics of the process $(Z,\mathbb{P}_z^{\uparrow})$ with $z\in (0,\infty)$. We start by establishing that it solves a certain martingale problem. 
\medskip


Recall $\ell=e^{\int_{0}^{x_0}\frac{2\Psi(u)}{cu}\ddr u}\in [0,\infty)$ and the functions $b,q,k$, defined in \eqref{eq:bphik}, such that for all $z>0$ and $y>0$,
\begin{equation*}
b(z)=z\frac{h'(z)}{h(z)},\quad q(z,y)=
\frac{z}{h(z)}\big(h(z+y)-h(z)\big) \text{ and } k(z)=\frac{c\ell}{2}\frac{z}{h(z)}.
\end{equation*}
Recall $\frac{\sigma^2}{2}$ the coefficient of the Feller part in the branching mechanism $\Psi$, see \eqref{eq:branchingmechanism}, and $\pi$ the branching Lévy measure. We have seen in Lemma \ref{lem:defh} that $h'(0)<\infty$ and $\int^{\infty}h(y)\pi(\ddr y)<\infty$. Recall the definition of $\mathcal{D}_Z$ in \eqref{DZ}. 
\begin{lemma}\label{lemma:generatorZup}\ 
For any $f$ such that $fh\in \mathcal{D}_Z$, one has for any $z>0$,
\begin{equation}\label{L+B}
\mathscr{L}^{\uparrow}f(z)=\mathscr{L}f(z)+\mathscr{B}f(z),
\end{equation}
with
\begin{align}\label{B}
\mathscr{B}f(z)&:=\sigma^2b(z) f'(z)+\int_0^{\infty}\left(f(z+y)-f(z)\right)q(z,y)\pi(\ddr y)-k(z)f(z).
\end{align}
Furthermore, the following limit exists and equals
\begin{align}\label{B(0)}
\mathscr{L}^{\uparrow}f(0):=\underset{z\rightarrow 0}{\lim}\, \mathscr{L}^{\uparrow}f(z)=\underset{z\rightarrow 0}{\lim}\, \mathscr{B}f(z)=\sigma^2f'(0)+\int_0^{\infty}\big(f(y)-f(0)\big)\frac{h(y)}{h'(0)}\pi(\ddr y)-\frac{c\ell}{2h'(0)}f(0).
\end{align}
In particular, $z\mapsto \mathscr{L}^{\uparrow}f(z)$ is bounded near $0$. 
\end{lemma}
\begin{proof}
The identities \eqref{L+B} and \eqref{B} follow from a direct application of Lemma \ref{lem:generalLfh} deferred to the Appendix. For the limit \eqref{B(0)}, recall $\mathscr{L}$, see  \eqref{genLCB}, and observe first that $\underset{z\rightarrow 0}{\lim}\, \mathscr{L}f(z)=\mathscr{L}f(0)=0$. From the definition of $\mathscr{B}f(z)$, with the functions $b$, $q$ and $k$ given in \eqref{eq:bphik} and the fact that $h'(0)<\infty$, we see that \[b(z)\underset{z\rightarrow 0}{\longrightarrow} 1,\ q(z,y)\underset{z \rightarrow 0}{\longrightarrow} \frac{h(y)}{h'(0)} \text{ and } k(z)\underset{z\rightarrow 0}{\longrightarrow}\frac{c\ell}{2h'(0)}.\] 
The limits denoted by $\mathscr{L}^{\uparrow}f(0)$ and $\mathscr{B}f(0)$ thus exist and coincide. \qed
\end{proof}
\begin{lemma}\label{lem:MPforZuparrow} For any $z\in (0,\infty)$, the process $(Z,\mathbb{P}_z^{\uparrow})$ satisfies the following martingale problem. For any $f$ such that $fh\in C^{2}([0,\infty))$, 
\begin{equation}\label{martingaleMPup} (M_t,t\geq 0):=\left(f(Z_t)-\int_{0}^{t}\mathscr{L}^{\uparrow}f(Z_s)\ddr s,t\geq 0\right)
\end{equation}
is a local martingale  under $\mathbb{P}_z^{\uparrow}$. The latter is a true martingale when $fh\in\mathcal{D}_Z$. 
\end{lemma}
\begin{proof}
Recall $\tau_n=n\wedge \zeta_{n}^{+}$. By applying Lemma \ref{lem:MPZ}, we see that for any $f$ such that $fh\in C^{2}([0,\infty))$,  \[\left(\frac{1}{h(z)}h(Z_{t\wedge \tau_n})f(Z_{t\wedge \tau_n})-\int_{0}^{t\wedge \tau_n}\frac{1}{h(z)}\mathscr{L}(fh)(Z_u)\ddr u,t\geq 0\right)\]
is a $\mathbb{P}_z$-martingale for all $z\in (0,\infty)$. By Fubini-Lebesgue's theorem, the definition of $\mathbb{P}^{\uparrow}_z$ and the fact that $\mathbb{P}^{\uparrow}_z(\zeta_0=\infty)=1$ (see Proposition \ref{prop:infimum}), one has
\begin{align*}
f(z)&=\mathbb{E}_z\left(\frac{1}{h(z)}h(Z_{t\wedge \tau_n})f(Z_{t\wedge \tau_n})-\int_{0}^{t\wedge \tau_n}\frac{\mathscr{L}(fh)(Z_u)}{h(z)}\ddr u\right)\\
&=\mathbb{E}_z\left(\frac{1}{h(z)}h(Z_{t\wedge \tau_n})f(Z_{t\wedge \tau_n})\right)-\int_{0}^{t}\mathbb{E}_z\left(\mathbbm{1}_{\{u\leq\tau_n\}}\frac{1}{h(z)}\mathscr{L}(fh)(Z_u)\right)\ddr u\\
&=\mathbb{E}^{\uparrow}_z\left(f(Z_{t\wedge \tau_n})\right)-\int_{0}^{t}\mathbb{E}_z\left(\mathbbm{1}_{\{u\leq\tau_n\}}\frac{1}{h(z)}\mathscr{L}(fh)(Z_u)\right)\ddr u\\
&=\mathbb{E}^{\uparrow}_z\left(f(Z_{t\wedge \tau_n})\right)-\int_{0}^{t}\mathbb{E}_z\left(\mathbbm{1}_{\{u\leq\tau_n\}}\frac{h(Z_u)}{h(z)}\frac{\mathscr{L}(fh)(Z_u)}{h(Z_u)}, u<\zeta_0\right)\ddr u\\
&=\mathbb{E}^{\uparrow}_z\left(f(Z_{t\wedge \tau_n})-\int_{0}^{t}\mathbbm{1}_{\{u\leq\tau_n\}}\mathscr{L}^{\uparrow}f(Z_u)\ddr u\right).
\end{align*}
Hence, we see, by using the Markov property under $\mathbb{P}^{\uparrow}_z$, that the process $M$ is a local martingale under $\mathbb{P}^{\uparrow}_z$ and the first statement is established. The fact that the local martingale $M$ is a martingale when $fh\in \mathcal{D}_Z$ stands from the fact that in this case, the functions $f$ and  $\mathscr{L}^{\uparrow}f$ are bounded (recall that the latter is always bounded near $0$). \qed
\end{proof}
\begin{lemma}\label{lem:uniquenessMPZup} Set
\begin{equation}\label{Duparrow} \mathcal{D}_{Z^{\uparrow}}:=\{f\in C^2: fh\in \mathcal{D}_Z\}.\end{equation}
There exists a unique, in law, Markov process, stopped at its first explosion time, solution to the martingale problem associated to $(\mathscr{L}^{\uparrow},\mathcal{D}_{Z^{\uparrow}})$.  
\end{lemma}
\begin{proof}
Existence has been established in Lemma \ref{lem:MPforZuparrow}. For the uniqueness, take two Markovian solutions $\mathbb{P}^1_z,\mathbb{P}^2_z$ of the martingale problem associated to $(\mathscr{L}^{\uparrow},\mathcal{D}_{Z^{\uparrow}})$. Let $(P^i_t), i\in \{1,2\}$ be their associated semigroups. Denote by $(P_t)$ the semigroup of the LCB process. Notice that $g=1/h$ is invariant for $\mathscr{L}^{\uparrow}$, indeed, $\mathscr{L}^{\uparrow}g=g\mathscr{L}hg=g\mathscr{L}1=0$. By reversing the arguments
of Lemma \ref{lem:MPforZuparrow} and examining the Doob's transform of $P_t^{i}$ with the function $g$,  we observe that either of the two solutions (stopped at their first explosion time) would satisfy the martingale problem $(\mathrm{MP})_Z$, see Section \ref{sec:martingaleproblemLCB}. The latter being well posed, we have for $i\in \{1,2\}$ and any $f\in \mathcal{D}_Z$, for any $z\in (0,\infty)$,
\begin{equation}\label{Pi}
\frac{1}{g(z)}P^i_t(gf)(z):=\frac{1}{g(z)}\mathbb{E}_z^i(g(Z_t)f(Z_t))=P_tf(z).
\end{equation} 
For any $x>0$, recall the notation $e_x:z\mapsto e^{-xz}$. Since $z\mapsto h(z)e^{-xz}$ is continuous vanishing at $\infty$ and
$\mathcal{D}_Z$ is dense in $C_0([0,\infty))$, one can find a sequence $(f_n)$ such that $f_n(z)\underset{n\rightarrow \infty}\longrightarrow h(z)e_x(z)$ uniformly in $z$. By replacing $f$ by $(f_n)$ in \eqref{Pi} and let $n$ go to $\infty$, we obtain $P^1_te_x(z)=P^2_te_x(z)$ for all $x>0,z\geq 0$ and therefore $\mathbb{P}^1_z=\mathbb{P}^2_z$.
\qed 
\end{proof}

\medskip
\textbf{Proof of Theorem \ref{thm:SDEZup}}. By Lemma \ref{lem:uniquenessMPZup}, the martingale problem associated to $(\mathscr{L}^{\uparrow},\mathcal{D}_{Z^{\uparrow}})$ has a unique Markovian solution. Theorem \ref{thm:SDEZup} follows by applying \cite[Theorem 2.3]{zbMATH05919793} which ensures that this solution is a weak solution to stochastic equation \eqref{SDEPuparrow}. The killing is seen here as a jump to $\infty$ occurring at rate $k$. The expression of $\zeta$ in terms of a standard exponential random variable $\mathbbm{e}$ readily follows.  The required condition in \cite[Theorem 2.3]{zbMATH05919793}, called (14), stems to the boundedness of the coefficients $b,q,k$  on compact sets, which is plainly checked. 

Uniqueness (up to explosion) of the weak solution of \eqref{SDEPuparrow} is verified as follows. Assume to the contrary that we have two different weak solutions. Then they are Markovian and solve the martingale problem associated to $(\mathscr{L}^{\uparrow},\mathcal{D}_{Z^{\uparrow}})$, which gives a contradiction to Lemma \ref{lem:uniquenessMPZup}. \qed 
\section{Proof of Theorem \ref{thm:entrance}}\label{sec:proofentrance}
We start by explaining the scheme of the proof. We recall that the process $Z$ under $\mathbb{P}_z^{\uparrow}$ is a $(0,\infty]$-valued Feller process with an almost surely  finite lifetime at which the process is absorbed at $\infty$. 
We will first establish that under \asp, the quantity $\mathbb{E}_z^{\uparrow}(e^{-xZ_t})$ admits a limit as $z\rightarrow 0$, denoted by $\mathbb{E}_{0}^{\uparrow}(e^{-xZ_t})$, for all $t>0$ and all $x>0$. Moreover, we shall see that $\mathbb{E}_{0}^{\uparrow}(e^{-xZ_t})<1$ for large enough $x$ and small enough $t>0$. This is the aim of Lemma \ref{lem:cvLT}. Auxiliary results on the Laplace dual diffusion $U$ will be needed and are gathered in Lemma \ref{lem:comparisonU}.  
\smallskip

In Lemma \ref{lem:entrancelaw}, we establish that the family of laws $\big(\mathbb{P}_z^{\uparrow}(Z_t\in \cdot),z>0\big)$ converges weakly towards a probability measure $\eta_t$ on $[0,\infty]$, as $z\rightarrow 0$. 
Next, by using $\eta_t$, we extend the semigroup $P_t^{\uparrow}$ at $0$, in a Fellerian way, so that for any bounded continuous function $f$,  \[P_t^{\uparrow}f(0):=P^{\uparrow}_tf(0+):=\underset{z\rightarrow 0+}{\lim} P^{\uparrow}_tf(z),\] and argue the existence of a probability law $\mathbb{P}^{\uparrow}_0$ on $\mathcal{D}$ such that \[\mathbb{P}^{\uparrow}_0(Z_0=0,\ \exists t>0: \forall s\geq t,\ Z_s>0)=1.\]

The framework of Feller processes will enable us also to appeal to some general results claiming that the family $(\mathbb{P}^{\uparrow}_z,z>0)$ converges towards $\mathbb{P}^{\uparrow}_0$, as $z$ goes to $0$, in the Skorokhod sense. 

Last, we will see that the process $(Z,\mathbb{P}^{\uparrow}_0)$ satisfies the martingale problem associated to the stochastic equation \eqref{SDEPuparrow}.
\smallskip

Recall the Laplace and Siegmund duality relationships, for any $x,z,y>0$ and $t\geq 0$: 
$$\mathbb{E}_z(e^{-xZ_t})=\mathbb{E}_x(e^{-zU_t}) \text{ and } \mathbb{P}_x(U_t<y)=\mathbb{P}_y(x<V_t).$$
We shall use the representation of the semigroup of $(Z,\mathbb{P}^{\uparrow}_z)$ given in \eqref{semigroupPtuparrow1}, in terms of the bidual process $V$ of the unconditioned LCB process $Z$. 
Recall that under Assumption \asp, the process  $U$ gets absorbed at $0$ in finite time almost surely.  
\medskip

We start with a preparatory lemma. The Laplace dual process $U$ satisfies \eqref{sdeU}, namely \begin{equation*} U_t=x+\int_{0}^{t}\sqrt{cU_s}\ddr B_s-\int_0^{t}\Psi(U_s)\ddr s,\end{equation*} with $B$ a Brownian motion. Since the function $\Psi$ is locally Lipschitz on $(0,\infty)$ and $u\mapsto \sqrt{cu}$ is 1/2-H\"older, the equation admits a unique strong solution $(U_s(x),s\geq 0)$, starting from $x>0$, up to the first hitting time of the boundaries, see e.g. \cite[Chapter IX, Section 3]{MR1725357}.
\begin{lemma}\label{lem:comparisonU} Assume \asp . 
\begin{enumerate}[label=\roman*)]
\item For any $x>0$ and $t\geq 0$, 
$\mathbb{E}_x(U_t)+\int_{0}^{t}\mathbb{E}_x(\Psi(U_s))\ddr s=x.$
\item For any $s\geq 0$ and any $w'\geq w$, almost surely,  $U_s(w')\geq U_s(w)$.
\item For sufficiently large $x$ and small enough $t$, one has:
$$\int_{0}^{t}\mathbb{E}\big(\Psi(U_s(x+v))-\Psi(U_s(x))\big)\ddr s>0, \text{ for all } v>0.$$
\end{enumerate} 
\end{lemma}
\begin{proof}
\begin{enumerate}
\item[i)] 
Recall that the process  $U$ starting from $x$ satisfies the stochastic differential equation  \eqref{sdeU}. 
The process $$(M_t,t\geq 0):=\left(U_t+\int_{0}^{t}\Psi(U_s)\ddr s, t\geq 0\right),$$ is thus a local martingale. 
We now verify that it is a true martingale. Recall that for any Lévy-Khintchine function $\Psi$, there are some constants $C_0,C_1>0$ such that $|\Psi(u)|\leq C_0+C_1u^2$ for all $u\geq 0$.  Let $\epsilon>0$. Set $T_\epsilon^{-}:=\inf\{t>0: U_t\leq \epsilon\}$. Assume for now established that the random variable $\sup_{s<t\wedge T_{\epsilon}^{-}}U_s^2$ is integrable, this ensures that the stopped process $(M_{t\wedge T_\epsilon^{-}},t\geq 0)$ is a martingale and therefore
\[\mathbb{E}_x(U_{t\wedge T_\epsilon^{-}})+\mathbb{E}_x\left(\int_{0}^{t\wedge T_\epsilon^{-}}\Psi(U_s)\ddr s\right)=x.\]
By letting $\epsilon$ go to $0$, we see that
\[\mathbb{E}_x(U_{t\wedge T_0})+\mathbb{E}_x\left(\int_{0}^{t\wedge T_0}\Psi(U_s)\ddr s\right)=x.\]
According to Lemma \ref{lem:Uexit}, under \asp, $0$ is an exit boundary, hence $U_{t\wedge T_0}=U_t$ for all $t\geq 0$ and recalling that $\Psi(0)=0$, we finally get the targeted equality
\[\mathbb{E}_x(U_{t})+\int_{0}^{t}\mathbb{E}_x(\Psi(U_s)\ddr s=x.\]
It remains to argue for the integrability of $\sup_{s<t\wedge T_{\epsilon}^{-}}U_s^2$. This can be checked as follows: for any $u\geq \epsilon$, $\Psi(u)\geq \frac{\Psi(\epsilon)}{\epsilon}u\geq -\gamma_\epsilon u$ for some $\gamma_\epsilon>0$. By the comparison theorem, see e.g. \cite[Theorem IX.3.7]{MR1725357},  $\mathbb{P}_x$-almost surely, for all $s\leq  T_\epsilon^{-}$, $U_{s}\leq \tilde{U}_{s}$ with \[\tilde{U}_{t}=x+\int_{0}^{t}\sqrt{c\tilde{U}_{s}}\ddr B_t+\gamma_\epsilon \int_{0}^{t}\tilde{U}_{s}\ddr s.\]
The process $(\tilde{U}_t,t\geq 0)$ is a supercritical CB process with branching mechanism $\tilde{\Psi}(q):=\frac{c}{2}q^2-\gamma_\epsilon q$. In particular this is a submartingale, with a second moment, and

\begin{equation}\label{supintegrable}\mathbb{E}_x\left(\sup_{s\leq t}U^2_{s\wedge T_\epsilon^{-}}\right)\leq \mathbb{E}_x\left(\sup_{s\leq t}\tilde{U}^2_{s}\right)\leq 4\mathbb{E}_x(\tilde{U}^2_{t})<\infty,
\end{equation}
where the second inequality is Doob's inequality applied to the submartingale $\tilde{U}$.

\item[ii)] For any $\epsilon>0$. Let $w'>w$ and recall $T^{-}_\epsilon(w)=\inf\{t>0: U_{t}(w)\leq \epsilon\}$. Recall that $\Psi$ is locally Lipschitz on $[\epsilon,\infty)$, by the comparison theorem, we have almost surely, for any $s\leq T^{-}_\epsilon(w)$, 
\[U_{s\wedge T^{-}_\epsilon(w)}(w')\geq U_{s\wedge T^{-}_\epsilon(w)}(w).\]
By letting $\epsilon$ goes to $0$, we see that
\[U_{s\wedge T_0(w)}(w')\geq U_{s\wedge T_0(w)}(w).\]
Since $0$ is an exit boundary for the diffusion $U$, we have $U_{s\wedge T_0(w)}(w)=U_{s}(w)$, and we get \[U_{s}(w')\geq U_{s}(w) \text{ a.s.} \]

\item[iii)] We are going to show that 
\begin{equation}\label{liminf0}\underset{s\rightarrow 0}{\liminf}\, \mathbb{E}\big[\Psi(U_s(x+v))-\Psi(U_s(x))\big]>0.
\end{equation}
Recall that $U_s(w)\rightarrow w$ a.s., as $s\rightarrow 0$, for all $w$. In the (sub)critical case, $\Psi$ is increasing on $[0,\infty)$. Let $x>x_0$ and $v>0$, by Fatou's lemma
\begin{align*}\underset{s\rightarrow 0}{\liminf}\, \mathbb{E}\big[\Psi(U_s(x+v))-\Psi(U_s(x))\big]&\geq \mathbb{E}\big[\underset{s\rightarrow 0}{\liminf}\left( \Psi(U_s(x+v))-\Psi(U_s(x)) \right) \big]\\
&=\Psi(x+v)-\Psi(x)>0.
\end{align*}
In the supercritical case, the assumption \asp, specifically $\Psi(\infty)=\infty$, ensures that $\Psi$ is strictly increasing on the interval $[x_0,\infty)$ with $x_0$ the largest zero of the function $\Psi$, thereby for all $u\in [0,x_0)$, $\Psi(u)\leq 0$. Let $x>x_0$ and $v>0$. One has
\begin{align}
&\underset{s\rightarrow 0}{\liminf}\, \mathbb{E}\big[\Psi(U_s(x+v))-\Psi(U_s(x))\big] \nonumber \\
&=\underset{s\rightarrow 0}{\liminf}\, \mathbb{E}\big[\big(\Psi(U_s(x+v))-\Psi(U_s(x))\big)\mathbbm{1}_{\{U_s(x+v)\leq x_0\}}\big]\label{liminf1}\\
&\qquad \qquad +\underset{s\rightarrow 0}{\liminf}\, \mathbb{E}\big[\big(\Psi(U_s(x+v))-\Psi(U_s(x))\big)\mathbbm{1}_{\{U_s(x+v)>x_0\}}\big]\label{liminf2}
\end{align}
The expression \eqref{liminf1} is bounded above by 
\[\mathbb{E}\big[|\Psi(U_s(x+v))-\Psi(U_s(x))|\mathbbm{1}_{\{U_s(x+v)\leq x_0\}}\big]\leq  2C_3\mathbb{P}(U_s(x+v)\leq x_0)\]
with $C_3:=\underset{u\in [0,x_0]}\sup|\Psi(u)|$. Since $x+v>x_0$, the upper bound tends to $0$ as $s$ goes to $0$. 

For the term \eqref{liminf2}, we first work on the event $\{U_s(x)> x_0\}$. By Fatou's lemma, everything being positive, the following limit is strictly positive:
\begin{align*}
&\underset{s\rightarrow 0}{\liminf}\, \mathbb{E}\big[\big(\Psi(U_s(x+v))-\Psi(U_s(x))\big)\mathbbm{1}_{\{U_s(x)>x_0\}}\big]\\
&\geq \mathbb{E}\big[\underset{s\rightarrow 0}{\liminf} \big(\Psi\left(U_s(x+v)\right)-\Psi\left(U_s(x)\right)\big)\mathbbm{1}_{\{U_s(x)>x_0\}}\big]\\
&= \Psi(x+v)-\Psi(x)>0.
\end{align*}
We now deal with the event $\{U_s(x)\leq x_0, U_s(x+v)>x_0\}$. Since $\Psi(U_s(x))\leq 0$, we have that
\[\Psi(U_s(x+v))-\Psi(U_s(x))\geq 0,\] and recalling that $x>x_0$, we get again by Fatou's lemma
\begin{align*}
&\underset{s\rightarrow 0}{\liminf}\, \mathbb{E}\big[\big(\Psi(U_s(x+v))-\Psi(U_s(x))\big)\mathbbm{1}_{\{U_s(x)\leq x_0,\, U_s(x+v)>x_0\}}\big]\geq 0.
\end{align*}
We have established  \eqref{liminf0}. 
%
Thus, for $x>x_0$ and $t$ small enough, one has for all $s\leq t$,  $\mathbb{E}\big(\Psi(U_s(x+v))-\Psi(U_s(x))\big)>0$  and 
$$\int_{0}^{t}\left(\mathbb{E}\big(\Psi(U_s(x+v))-\Psi(U_s(x))\big)\right)\ddr s>0.
$$
\qed
\end{enumerate}
\end{proof}
In the next Lemma, we study the convergence of the Laplace transform of the process under $\mathbb{P}_z^{\uparrow}$ as $z$ goes to $0$.
\begin{lemma}\label{lem:cvLT} Assume \asp .
\begin{enumerate}
\item[i)] For any $x>0$, $t>0$, the following limit exists:
\begin{center}
$\mathbb{E}^{\uparrow}_{0}(e^{-xZ_t}):=\underset{z\rightarrow 0}{\lim} \mathbb{E}^{\uparrow}_z(e^{-xZ_t})$.\end{center}
\item[ii)] For any $x>0$,
$\mathbb{E}^{\uparrow}_{0}(e^{-xZ_t})\underset{t\rightarrow 0}{\longrightarrow} 1$.
\item[iii)] Moreover, there exist $x_1>0$ and $t_1>0$ such that
\begin{center}  $\forall x>x_1$, $\forall t\in (0,t_1)$, $\mathbb{E}^{\uparrow}_{0}(e^{-xZ_t})<1$.
\end{center}
\end{enumerate}

\end{lemma}
\begin{proof}
Recall from \eqref{semigroupPtuparrow1} that for any  $z,x,t>0$,
\begin{align}\label{eq:represPtupproof}
\mathbb{E}_z^{\uparrow}[e^{-xZ_t}]&=\frac{z}{h(z)}\int_{0}^{\infty}e^{-zy}\mathbb{E}_y\left[\mathbbm{1}_{\{x\leq V_t\}}S(V_t-x)\right]\ddr y. 
\end{align}
\begin{enumerate}
\item[i)] Recall that $h(0)=0$ and that $h$ admits a finite derivative at $0$, $h'(0)=\int_{0}^{\infty}S(v)\ddr v<\infty$. Thus, $z/h(z)\underset{z\rightarrow \infty}{\longrightarrow} 1/h'(0)<\infty$. On the other hand, for any fixed $t>0$, the following limit, as $z$ goes to $0$, exists by monotonicity,
\[\underset{z\rightarrow 0}{\lim}\int_{0}^{\infty}e^{-zy}\mathbb{E}_{y}\left(\mathbbm{1}_{\{x\leq V_t\}}S(V_t-x)\right)\ddr y=\int_{0}^{\infty}\mathbb{E}_{y}\left(\mathbbm{1}_{\{x\leq V_t\}}S(V_t-x)\right)\ddr y.\]
This ensures that \eqref{eq:represPtupproof} also converges as $z$ goes to $0$.
\item[ii)]  Since the expression \eqref{eq:represPtupproof} is bounded by $1$, we have
\begin{equation}\label{ineq}
\int_{0}^{\infty}\mathbb{E}_{y}\left(\mathbbm{1}_{\{x\leq V_t\}}S(V_t-x)\right)\ddr y\leq \int_{0}^{\infty}S(v)\ddr v=h'(0).\end{equation}
By Fatou's Lemma, we see that 
\begin{align}\label{eq:threeineq} 
\int_{0}^{\infty}\underset{t\rightarrow 0}{\liminf}\, \mathbb{E}_{y}\left(\mathbbm{1}_{\{x\leq V_t\}}S(V_t-x)\right)\ddr y&\leq \underset{t\rightarrow 0}{\liminf}\int_{0}^{\infty}\mathbb{E}_{y}\left(\mathbbm{1}_{\{x\leq V_t\}}S(V_t-x)\right)\ddr y\leq h'(0).
\end{align}
Note that for any $y\in (0,\infty)$, $\mathbb{E}_{y}\big(\mathbbm{1}_{\{x\leq V_t\}}S(V_t-x)\big)$ converges towards $\mathbbm{1}_{\{x\leq y\}}S(y-x)$ pointwise as $t$ goes to $0$. Moreover, 
$$\int_{0}^{\infty}\mathbbm{1}_{\{x\leq y\}}S(y-x)\ddr y=\int_{0}^{\infty}S(v)\ddr v=h'(0).$$ 
We deduce that the three terms in the inequality \eqref{eq:threeineq} are equal. Then, by applying \eqref{ineq}, we get
\begin{align*}h'(0)&=\underset{t\rightarrow 0}{\liminf} \int_{0}^{\infty}\mathbb{E}_{y}\left(\mathbbm{1}_{\{x\leq V_t\}}S(V_t-x)\right)\ddr y\\
&\leq \underset{t\rightarrow 0}{\limsup} \int_{0}^{\infty}\mathbb{E}_{y}\left(\mathbbm{1}_{\{x\leq V_t\}}S(V_t-x)\right)\ddr y\leq h'(0),
\end{align*}
and finally, for any $x>0$, one has
\[\mathbb{E}^{\uparrow}_{0}(e^{-xZ_t})=\frac{1}{h'(0)}\int_{0}^{\infty}\mathbb{E}_{y}\left(\mathbbm{1}_{\{x\leq V_t\}}S(V_t-x)\right)\ddr y\underset{t\rightarrow 0}{\longrightarrow} 1.\]
\item[iii)]
We now establish that $\mathbb{E}^{\uparrow}_{0}(e^{-xZ_t}):=\underset{z\rightarrow 0}{\lim}\mathbb{E}_z^{\uparrow}(e^{-xZ_t})<1$ for some $x,t>0$. Plainly, the following equivalence holds:
\begin{center}
$\mathbb{E}_{0}^{\uparrow}(e^{-xZ_t}) <1$
if and only if the inequality \eqref{ineq} is strict. 
\end{center}
The goal is reached if and only if
\[C_t(x):=h'(0)-\int_{0}^{\infty}\mathbb{E}_{y}\left(\mathbbm{1}_{\{x\leq V_t\}}S(V_t-x)\right)\ddr y>0, \text{ for some } t>0.\] 
By Fubini-Tonelli's theorem and Siegmund duality, 
\begin{align*}
\mathbb{E}_y\big(\mathbbm{1}_{\{x\leq V_t\}}S(V_t-x)\big)&=\mathbb{E}_y\left(\int_{0}^{\infty}(-S'(v))\mathbbm{1}_{\{x\leq V_t\leq v+x\}}\ddr v\right)\\
&=\int_{0}^{\infty}(-S'(v))\big(\mathbb{P}_{x}(U_t\leq y)-\mathbb{P}_{x+v}(U_t\leq y)\big)\ddr v.
\end{align*}
Hence the left-hand side of \eqref{ineq} can be rewritten as follows
\begin{align}
\int_{0}^{\infty}\ddr y \mathbb{E}_y\big(\mathbbm{1}_{\{x\leq V_t\}}S(V_t-x)\big)&=\int_{0}^{\infty}(-S'(v))\ddr v\, \mathbb{E}\big(U_t(x+v)-U_t(x)\big) \label{line1}\\
&=\int_{0}^{\infty}(-S'(v))\ddr v \left(v-\int_{0}^{t}\mathbb{E}\big(\Psi(U_s(v+x))-\Psi(U_s(x))\big)\ddr s\right) \label{line2} \\
&=h'(0)-\int_{0}^{\infty}(-S'(v))\ddr v \int_{0}^{t}\mathbb{E}\big(\Psi(U_s(v+x))-\Psi(U_s(x))\big)\ddr s, \label{line3}
\end{align}
where we used in \eqref{line2}, the identity in Lemma \ref{lem:comparisonU}-i),
\[\mathbb{E}\big(U_t(x+v)\big)=x+v-\int_{0}^{t}\mathbb{E}\big(\Psi(U_s(x+v))\big)\ddr s,\]
and in \eqref{line3}, $h'(0)=\int_{0}^{\infty}\big(-S'(v)\big)v\ddr v$.
%
By Lemma \ref{lem:comparisonU}-iii),  the integrand in \eqref{line3} is strictly positive for $x$ large enough and $t$ small enough, therefore $C_t(x)>0$ and the inequality \eqref{ineq} is strict for those $x$ and $t$. This achieves the proof of item iii).
\qed
\end{enumerate}
\end{proof}
\begin{lemma}\label{lem:entrancelaw}
For all $t>0$, there exists a probability law $\eta_t$ on $[0,\infty]$ such that $\eta_t\neq \delta_0$, and the following weak convergence holds
\begin{equation}\label{cvtoetat}P^{\uparrow}_t(z,\ddr y):=\mathbb{P}^{\uparrow}_z(Z_t\in \ddr y)\underset{z\rightarrow 0}{\longrightarrow} \eta_t(\ddr y).
\end{equation} 
Moreover, $\eta_t\underset{t\rightarrow 0}{\longrightarrow} \delta_0$, and the following extension of $P_t^{\uparrow}$ on $[0,\infty]$ defines a $C_0$-Feller semigroup: 
\begin{equation}\label{eq:extensionsemigroup}
\forall f\in C([0,\infty]),\quad 
P_t^{\uparrow}f(z):= \begin{cases} P_t^{\uparrow}f(z), & \text{ if } z\in (0,\infty],\\ 
\int_{[0,\infty]}f(y)\eta_t(\ddr y), & \text{ if } z=0.
\end{cases}
\end{equation}
\end{lemma}
\begin{proof}
We have established in Lemma \ref{lem:cvLT}, the pointwise convergence of the Laplace transforms, $\mathbb{E}_z^{\uparrow}(e^{-xZ_t})$ as $z$ goes to $0$ towards the function $x\mapsto \mathbb{E}_{0}^{\uparrow}(e^{-xZ_t})$. 
By Lévy's continuity theorem, there exists a probability measure $\eta_t$ on $[0,\infty]$,  such that \begin{equation}\label{eq:LTweakcv}
\mathbb{E}_{0}^{\uparrow}(e^{-xZ_t})=\int_{[0,\infty]}e^{-xz}\eta_t(\ddr z) \text{ and } 
\mathbb{P}^{\uparrow}_z(Z_t\in \cdot)\underset{z\rightarrow 0}{\longrightarrow} \eta_t(\cdot).
\end{equation}
The fact that $\eta_t\rightarrow \delta_0$, as $t \rightarrow 0$, is a direct consequence of the convergence of the Laplace transform of $\eta_t$ towards the constant function $1$ obtained in Lemma \ref{lem:cvLT}-ii). Let $f\in C_0([0,\infty))$, recall that the semigroup $P^{\uparrow}_t$ is Feller on $(0,\infty]$, so that $P^{\uparrow}_tf\lvert_{(0,\infty]}\in C_0$ for any $f\in C_0([0,\infty))$. The continuity at $0$ holds by definition of $P_t^{\uparrow}f(0)$. Plainly, \begin{equation}\label{eq:Fellerextension}
P_t^{\uparrow}C_0([0,\infty))\subset C_0([0,\infty)), \text{ and for all } z\geq 0, P^{\uparrow}_tf(z)\underset{t\rightarrow 0}{\longrightarrow} f(z).\end{equation}
It remains to verify that $P^{\uparrow}_t$, extended at $0$, is a semigroup. The latter will then be Fellerian, by \eqref{eq:Fellerextension}. By the weak convergence in \eqref{eq:LTweakcv}, the fact that $P_t^{\uparrow}(z,\{0\})=0$ for all $z>0$, see Proposition~\ref{prop:infimum}, together with the semigroup property, on $(0,\infty]$, i.e., for all $s>0$, $t\geq 0$, $y\in (0,\infty]$, $P^{\uparrow}_{s+t}f(y)=P^{\uparrow}_{s}\big(P^{\uparrow}_{t}f\big)(y)$, we have that
\begin{align*}
P_s^{\uparrow}(P^{\uparrow}_t)f(0)&=\int_{[0,\infty]}\eta_s(\ddr y)P^{\uparrow}_tf(y)\\
&=\underset{z\rightarrow 0}{\lim} \int_{(0,\infty]}P^{\uparrow}_s(z,\ddr y)P^{\uparrow}_tf(y)\\
&=\underset{z\rightarrow 0}{\lim}\, P^{\uparrow}_{s+t}f(z)=P^{\uparrow}_{s+t}f(0).
\end{align*}
This ensures that the extension of $P_t^{\uparrow}$ defined in \eqref{eq:extensionsemigroup} is a semigroup. Furthermore, Lemma \ref{lem:cvLT}-iii) ensures that there exist a small time $t_0>0$ and a large enough $x$ such that $\int_{[0,\infty]}e^{-xy}\eta_t(\ddr y)<1$ for all $t\in (0,t_0]$. This entails that  $\eta_t(\ddr y)\neq \delta_0$ for $t\in (0,t_0]$. By the semigroup property, we see that for all $s\geq 0$ and $t\in (0,t_0]$, $$P^{\uparrow}_{s+t}f(0)\neq P^{\uparrow}_tf(0)\neq f(0).$$
Thus, $\eta_t\neq \delta_0$ holds true for all $t>0$. 
\qed
\end{proof}
\begin{lemma}\label{lem:finalcv} There exists a unique probability measure $\mathbb{P}^{\uparrow}_0$ on $\mathcal{D}$ such that $\mathbb{P}^{\uparrow}_0(Z_0=0)=1$ under which the canonical process $Z$ is Feller with semigroup $(P^{\uparrow}_t)$.
Moreover, as $z$ goes to $0$, $\mathbb{P}^{\uparrow}_z \Longrightarrow \mathbb{P}^{\uparrow}_0$, in Skorokhod sense.
\end{lemma}
\begin{proof}
A standard theorem, see e.g. \cite[Theorem III-(2.7), p.91]{MR1725357}, guarantees that there is a family of probability measures on $\mathcal{D}$, $(\mathbb{P}^{\uparrow}_z ,z\in [0,\infty])$, which are the laws of a càdlàg Feller process, with semigroup $(P^{\uparrow}_t)$, starting from $z\in [0,\infty]$.
We first check that the law $\mathbb{P}^{\uparrow}_z$ converges towards $\mathbb{P}^{\uparrow}_0$ as $z$ goes to $0$ in the finite-dimensional sense. This is a consequence of the convergence \eqref{cvtoetat}, together with the Feller semigroup property. Indeed, let $n\geq 1$, and $f_1,\cdots,f_n$ be continuous functions on $[0,\infty]$, then
\begin{align}\label{fdd}
\mathbb{E}^{\uparrow}_z\big[f_1(Z_{t_1})&f_2(Z_{t_2})\cdots f_n(Z_{t_n})\big]\\ \nonumber 
&=\int_{[0,\infty]}P_{t_1}^{\uparrow}(z,\ddr z_1)f_1(z_1)\int_{[0,\infty]^{n-1}}P_{t_2-t_1}^{\uparrow}(z_1,\ddr z_2)f_2(z_2)\cdots P_{t_n-t_{n-1}}^{\uparrow}(z_{n-1},\ddr z_n)f_n(z_n).
\end{align}
Since the integrand above is continuous in $z_1$ by the Feller property and by continuity under the integral, we see from \eqref{cvtoetat}, that the left hand side in \eqref{fdd} converges towards $\mathbb{E}^{\uparrow}_0\big[f_1(Z_{t_1})f_2(Z_{t_2})\cdots f_n(Z_{t_n})\big]$ as $z$ goes to $0$. Once the convergence of finite-dimensional laws observed, tightness of the latter can be established using Aldous' criterion, we refer for instance to Foucart et al. \cite[Theorem 2.5, Lemma 3.1]{zbMATH07242423}, see in particular the proof of Lemma 3.1.
\qed
\end{proof}
\begin{lemma} The process $(Z,\mathbb{P}_0^{\uparrow})$ leaves the boundary $0$ and never returns almost surely.
\end{lemma}
\begin{proof}
We start by showing that $0$ is not absorbing. Let $t>0$ be fixed. By iterating the Markov property, we get for all $n\geq 2$,
\begin{align*}
\mathbb{P}^{\uparrow}_0(Z_t=0,Z_{2t}=0,\cdots, Z_{nt}=0)&=\mathbb{P}_0(Z_t=0,Z_{2t}=0,\cdots, Z_{(n-1)t}=0)\mathbb{P}^{\uparrow}(Z_{nt}=0|Z_{(n-1)t}=0)\\
&=\mathbb{P}^{\uparrow}_0(Z_t=0,Z_{2t}=0,\cdots, Z_{(n-1)t}=0)\mathbb{P}^{\uparrow}_{0}(Z_t=0)=\eta_t(0)^{n}.
\end{align*}
Since $\eta_t\neq \delta_0$, see Lemma \ref{lem:entrancelaw},  $\eta_t(0)<1$ and we see that
\[\underset{n\rightarrow \infty}{\lim} \mathbb{P}^{\uparrow}_0(Z_t=0,Z_{2t}=0,\cdots, Z_{nt}=0)=\mathbb{P}^{\uparrow}_0(\forall m\geq 1, Z_{mt}=0)=0.\]
Thus, $Z_{mt}>0$ for some $m\geq 1$, $\mathbb{P}^{\uparrow}_0$-almost surely. For all $r>0$, denote the first passage time above level $r$ by $\zeta^+_r:=\inf\{t>0: Z_t>r\}$, we just have established that the event $\bigcup_{r>0}\{\zeta^+_r<\infty\}$ has probability one. On $\{\zeta^+_r<\infty\}$, by the strong markov property at $\zeta^+_r$, together with the fact that $\mathbb{P}^{\uparrow}_{z}(\zeta_0=\infty)=1$ for any $z>0$, see Proposition \ref{prop:infimum},  we see that $\mathbb{P}^{\uparrow}_0$-a.s., $(Z_{t+\zeta^+_r},t\geq 0)$ never hits $0$. This concludes the proof.
\end{proof}
\begin{lemma}\label{lem:sdefrom0} The process $(Z,\mathbb{P}^{\uparrow}_0)$ is a weak solution to the stochastic equation \eqref{SDEPuparrow} starting from~$0$.
\end{lemma}
\begin{proof} 
Let $f\in \mathcal{D}_{Z^{\uparrow}}$, see \eqref{Duparrow}. Recall $\mathscr{L}^{\uparrow}f$ in \eqref{L+B} and the martingale $M$ defined in \eqref{martingaleMPup} as:
\[M_t:=f(Z_t)-\int_{0}^{t}\mathscr{L}^{\uparrow}f(Z_s)\ddr s.\]
Plainly for any bounded continuous function $G$ and all $s\leq t$,
\begin{equation}\label{eqmtgG}
\mathbb{E}_z^{\uparrow}\big((M_t-M_s)G(Z_u,u\leq s)\big)=0.
\end{equation}
Recall that $\mathscr{L}^{\uparrow}f(z)\rightarrow \mathscr{L}^{\uparrow}f(0)$, as $z$ goes to $0$, see \eqref{B(0)}. Since $\mathbb{P}_z^{\uparrow}$ converges towards  $\mathbb{P}_0^{\uparrow}$ in Skorokhod sense, we get by letting $z$ go to $0$ in \eqref{eqmtgG}, that
\[\mathbb{E}_0^{\uparrow}\big((M_t-M_s)G(Z_u,u\leq s)\big)=0.\]
Thus $(M_t,t\geq 0)$ is a $\mathbb{P}_0^{\uparrow}$-martingale. Therefore $(Z,\mathbb{P}_0^{\uparrow})$ satisfies the martingale problem associated to the stochastic equation \eqref{SDEPuparrow} with initial value $0$. This achieves the proof. \qed \end{proof}
\smallskip
\noindent \textbf{Proof of Theorem \ref{thm:entrance}}. It follows by combining Lemma \ref{lem:finalcv} and Lemma \ref{lem:sdefrom0}. 
\qed
\section{CB case: proof of Theorem \ref{thmCSBP} and Proposition \ref{propcbIfrom0}}\label{sec:prooftheoremCSBP}
We treat here the  setting of (sub)-critical CB processes, $\varrho=\Psi'(0+)\geq 0$. Arguments are similar but simpler and we omit details.
\subsection{Proof of Theorem \ref{thmCSBP}}
Recall $J=\int_{0}^{\infty}Z_s\ddr s$. By  \cite[Proposition 2.3]{BINGHAM1976217},
\[\mathbb{E}_z(e^{-\theta J})=e^{-z\Phi(\theta)},\]
with $\Phi(\theta)=\Psi^{-1}(\theta)$. Moreover, since $\Phi(\theta)$ goes to $0$ as $\theta$ goes to $0$, one has
\[\frac{1}{\Phi(\theta)}\mathbb{E}_z(1-e^{-\theta J})=\frac{1}{\Phi(\theta)}(1-e^{-z\Phi(\theta)})\underset{\theta \rightarrow 0}{\longrightarrow} z.\]
Let $\mathbbm{e}$ be an independent standard exponential random variable and set $\mathbbm{e}_\theta:=\mathbbm{e}/\theta$ for any $\theta>0$. The arguments designed for establishing Lemma \ref{lem:conditioningalongprogeny}  hold verbatim. We only sketch the arguments. 

Set $J_t=\int_{0}^{t}Z_s\ddr s$. By \eqref{calculation1} and \eqref{calculation2}, we have for any $\Lambda\in \mathcal{F}_t$,
\begin{align*}
\mathbb{P}_z(\Lambda,t<\mathbbm{e}_\theta|J\geq \mathbbm{e}_\theta)&=\mathbbm{E}_z\left(\mathbbm{1}_{\{\Lambda,t<\mathbbm{e}_\theta\}}e^{-\theta J_t}\frac{\mathbb{E}_{Z_t}(1-e^{-\theta J})}{\mathbb{E}_{z}(1-e^{-\theta J})}\right)\underset{\theta \rightarrow 0}{\longrightarrow} \mathbbm{E}_z\left(\mathbbm{1}_{\Lambda}\frac{Z_t}{z}\right).
\end{align*}
Therefore
$$\mathbb{P}^{\uparrow}_z(\Lambda,t<\zeta)=\underset{\theta \rightarrow 0}{\lim}\, \mathbb{P}_z(\Lambda,t<\mathbbm{e}_\theta|J\geq \mathbbm{e}_\theta)=\mathbb{E}_z\left(\frac{Z_t}{z}\mathbbm{1}_{\Lambda}\right).$$
Denote by $\mathscr{L}$ the generator of the CB process, i.e. $\mathscr{L}f(z)=z\mathrm{L}^{\Psi}f(z)$, see Section \ref{sec:martingaleproblemLCB}. In the (sub)-critical case, $\mathscr{L}$ takes the following form
\[\mathscr{L}f(z)=z\left(\frac{\sigma^2}{2}f''(z)-\rho f'(z)+\int_{0}^{\infty}\left(f(z+y)-f(z)-yf'(z)\right)\pi(\ddr y)\right).\]
In particular, by setting $h(z):=z$, we see that $\mathscr{L}h(z)=-\rho h(z)$ and by Lemma \ref{lem:generalLfh}, the process $(Z,\mathbb{P}^{\uparrow}_z)$ satisfies the martingale problem, associated to $\big(\mathscr{L}^{\uparrow},\mathcal{C}_c^2(0,\infty)\big)$ with $\mathscr{L}^{\uparrow}=\mathscr{L}+\mathscr{B}$, where
\[\mathscr{B}f(z)=\sigma^2f'(z)+\int_{0}^{\infty}(f(z+y)-f(z)y\pi(\ddr y)-\rho f(z).\]
It is a weak solution to the stochastic equation \eqref{conditioned partcsbp} and its law is that of a CBI$(\Psi,\Psi')$,  see  Kawazu and Watanabe \cite{KAW} and Li's book \cite[Chapters 9, 10]{zbMATH07687769} for a general study of CBI processes\footnote{the setting with a killing term in the subordinator is not treated there but causes no problem}. 
\qed
\subsection{Proof of Proposition \ref{propcbIfrom0}}
For any $z>0$, the process $(Z,\mathbb{P}^{\uparrow}_z)$ being a CBI$(\Psi,\Psi')$ process, its semigroup satisfies: for any $x>0,t\geq 0$,
\[\mathbb{E}_z^{\uparrow}(e^{-xZ_t})=e^{-zu_t(x)-\int_{0}^{t}\Psi'(u_s(x))\ddr s},\]
with $\frac{\ddr}{\ddr t} u_t(x)=-\Psi\big(u_t(x)\big), u_0(x)=x$, see e.g. \cite[Equation 12.25, page 354]{MR3155252}. Hence, by letting $z$ go to $0$, we see that
\[\mathbb{E}_z^{\uparrow}(e^{-xZ_t})\underset{z\rightarrow 0+}{\longrightarrow} \mathbb{E}_{0}^{\uparrow}(e^{-xZ_t}):=e^{-\int_{0}^{t}\Psi'(u_s(x))\ddr s}.\]
Plainly, since $\Psi'>0$ on $(0,\infty)$ and $u_t(x)\in (0,\infty)$, $\mathbb{E}_{0}^{\uparrow}(e^{-xZ_t})<1$. Similarly as in Section~\ref{sec:proofentrance}, see Lemma \ref{lem:entrancelaw} and Lemma \ref{lem:finalcv}, one can define $\mathbb{P}_{0}^{\uparrow}$ and we have that $\mathbb{P}_z^{\uparrow}\Longrightarrow \mathbb{P}_{0}^{\uparrow}$. The probability measure $\mathbb{P}_{0}^{\uparrow}$ is nothing but the law of the CBI process started from $0$. \qed
\begin{remark} 
By using the fact that $(e^{\rho t}Z_t,t\geq 0)$ is a $\mathbb{P}_z$-martingale,  
one plainly sees that $(Z,\mathbb{P}^{\uparrow}_z)$ has the same law as a CBI$(\Psi,\Psi'-\rho)$, killed at an exponential time with parameter $\rho$. 
Last, since $\Psi$ is (sub)-critical, the CBI$(\Psi,\Psi'-\rho)$ does not explode, see e.g. \cite[Theorem 1.2]{KAW}, and the process $(Z,\mathbb{P}^{\uparrow}_z)$, when $\rho>0$, is killed almost surely, i.e. $\mathbb{P}^{\uparrow}_z(Z_{\zeta-}<\infty)=1$. By Remark \ref{rem:killing}, the process $(Z_{\zeta_n^+\wedge n},n\geq 0)$, under $\mathbb{P}_z$, is then  uniformly integrable. 
\end{remark}
\appendix
\section{Proofs of Proposition \ref{lem:joiningduals1} and Lemma \ref{lemma:generatorZup}}
\subsection{Generator's of $h$-transformed processes}\label{sec:proofgenhprocess}
We establish here Lemma \ref{lemma:generatorZup} in which an explicit expression of the operator $\mathscr{L}^{\uparrow}f=\frac{1}{h}\mathscr{L}(fh)$ is found.  The result is of general flavour as it holds for any function $h$ and any generator $\mathscr{L}$. We state it therefore in the general framework of an operator $\mathscr{L}$  of Courrège-Von Waldenfels form, see e.g. B\"ottcher et al. \cite[Theorem 2.21]{levy-matters-III}, namely:
$$\mathscr{L}f(z)=\mathscr{A}f(z)+\mathscr{J}f(z),$$
with $\mathscr{A}$ the local part and $\mathscr{J}$ the jump part, that is to say $$\mathscr{A}f(z):=a(z)f''(z)+b(z)f'(z)
\text{ and } \mathscr{J}f(z):=\int_{\mathbb{R}}\left(f(z+y)-f(z)-yf'(z)\mathbbm{1}_{\{y<1\}}\right)\nu(z,\ddr y),$$
for $a,b$ two functions ($a$ is nonnegative) and some Lévy kernel $(\nu(z,\ddr y),z\in \mathbb{R})$ such that for any $z\in \mathbb{R}$,  $\nu(z,\{0\})=0$  and $\int_{\mathbb{R}}(1\wedge y^2)\nu(z,\ddr y)<\infty$. 

The statement below is not new, see for instance Kunita \cite{zbMATH03297712}. We provide a quick proof.

\begin{lemma}\label{lem:generalLfh} Let $h$ and $f$ be such that $hf\in C^{2}([0,\infty))$ and $\int^{\infty}|f(y)h(y)|\pi(\ddr y)<\infty$. For any $z>0$,
$$\mathscr{L}^{\uparrow}f(z)=\mathscr{L}f(z)+\mathscr{B}f(z),$$
with 
\begin{align}\label{Bappendix}
\mathscr{B}f(z)&:=\frac{\mathscr{L}(hf)(z)-h(z)\mathscr{L}f(z)}{h(z)} \nonumber \\
&=a(z)\frac{h'(z)}{h(z)}f'(z)+\int_{\mathbb{R}}\big(f(z+y)-f(z)\big)\frac{h(z+y)-h(z)}{h(z)}\nu(z,\ddr y)+\frac{\mathscr{L}h(z)}{h(z)}f(z).
\end{align}
\end{lemma}
\begin{proof}
We compute $\mathscr{L}fh(z)$. We first consider the local part. One has 
\begin{align}\label{localpart}
\mathscr{A}fh(z)&=b(z)\big(h'(z)f(z)+h(z)f'(z)\big)+\frac{a(z)}{2}\big(h''(z)f(z)+2h'(z)f'(z)+h(z)f''(z)\big) \nonumber \\
&=h(z)\mathscr{A}f(z)+f(z)\mathscr{A}h(z)+a(z)h'(z)f'(z).
\end{align}
For the jump part, one can plainly check the following identity:
\begin{align*}
f(z+y)h(z+y)-f(z)h(z)&-y(fh)'(y)\mathbbm{1}_{\{y<1\}}\\
&=h(z)\big(f(z+y)-f(z)-yf'(z)\mathbbm{1}_{\{y<1\}}\big)\\
&\qquad \qquad +\big(h(z+y)-h(z)\big)\big(f(z+y)-f(z)\big)\\
&\qquad \qquad \qquad +f(z)\big(h(z+y)-h(z)-yh'(z)\mathbbm{1}_{\{y<1\}}\big).
\end{align*}
By integrating  with respect to $\nu(z,\ddr y)$ both sides of the equality above, we get 
\begin{equation}\label{jumppart}
\mathscr{J}fh(z)=h(z)\mathscr{J}f(z)+\int_{0}^{\infty}\big(h(z+y)-h(z)\big)\big(f(z+y)-f(z)\big)\nu(z,\ddr y)+f(z)\mathscr{J}h(z).
\end{equation}
Finally, we obtain by adding \eqref{localpart} and \eqref{jumppart},
\[\frac{\mathscr{L}fh(z)}{h(z)}=\frac{1}{h(z)}\left(\mathscr{A}fh(z)+\mathscr{J}fh(z)\right)=\mathscr{A}f(z)+\mathscr{J}f(z)+\mathscr{B}f(z)=\mathscr{L}f(z)+\mathscr{B}f(z).\]
\qed
\end{proof}

\textbf{Proof of Lemma \ref{lemma:generatorZup}}. This is a direct consequence of Lemma \ref{lem:generalLfh}. Recall the generator of the LCB process in \eqref{genLCB}, in this setting we have for all $z\geq 0$,
\begin{center}
$a(z)=\frac{\sigma^2}{2}z$, $b(z)=-\gamma z-\frac{c}{2}z^2$ and $\nu(z,\ddr y)=z\pi(\ddr y)$.
\end{center}
\smallskip
The form of $\mathscr{B}$ in Lemma \ref{lemma:generatorZup} is easily deduced by replacing each term in \eqref{Bappendix} by the ones above. By Lemma \ref{lem:Lh}, we see that the killing term $\frac{\mathcal{L}h(z)}{h(z)}$ is given by $k(z)=-\frac{\mathscr{L}h(z)}{h(z)}=\frac{c\ell}{2}\frac{z}{h(z)}$. \qed
\subsection{The bidual process $V$: proof of Proposition \ref{lem:joiningduals1}}\label{sec:proofpropbidual}
The identity \eqref{joiningduals} is shown as follows. Denote by $\mathbbm{e}_z$ an exponential random variable with parameter $z$, independent from process $U$. One has by the Laplace duality, see Lemma \ref{lemmadualityLA}, and the Siegmund duality, see Lemma~\ref{lem:siegmund},
\[\mathbb{E}_z(e^{-xZ_t})=\mathbb{E}_x(e^{-zU_t})=\mathbb{P}_x(\mathbbm{e}_z>U_t)=\int_{0}^{\infty}ze^{-zy}\mathbb{P}_x(y>U_t)\ddr y=\int_{0}^{\infty}ze^{-zy}\mathbb{P}_y(V_t>x)\ddr y.\]
For any $t,y,v\geq 0$, set $F_{t}(v,y):=\mathbb{P}_v(V_t\leq y)$ and
$y\mapsto f_t(v,y):=\frac{\ddr }{\ddr y}F_{t}(v,y)$. By derivation under expectation, one has for all $y,z>0$  and any $t\geq 0$,\begin{align*}
\mathbb{E}_z(e^{-xZ_t})=\mathbb{E}_x(e^{-zU_t})=\mathbb{P}_x(\mathbbm{e}_z>U_t)&=\int_{0}^{\infty}ze^{-zy}\mathbb{P}_x(y>U_t)\ddr y\\
&=\int_{0}^{\infty}ze^{-zy}\mathbb{P}_y(V_t>x)\ddr y.
\end{align*}
For any $t,y,v\geq 0$, set $F_{t}(v,y):=\mathbb{P}_v(V_t\leq y)$ and
$y\mapsto f_t(v,y):=\frac{\ddr }{\ddr y}F_{t}(v,y)$. By derivation under expectation, one has for all $y,z>0$  and any $t\geq 0$, \begin{equation}\label{derivative}
\mathbb{E}_z[Z_te^{-yZ_t}]=\int_{0}^{\infty}ze^{-zv}f_t(v,y)\ddr v.
\end{equation}
By using the Markov property repeatedly
\begin{align*}
&\mathbb{E}_z\left[e^{-x_1Z_{t_1}}\cdots e^{-x_{n-1}Z_{t_{n-1}}}(1-e^{-x_nZ_{t_n}})\right]\\
&=\mathbb{E}_z\left[e^{-x_1Z_{t_1}}\cdots e^{-x_{n-1}Z_{t_{n-1}}}\mathbb{E}_{Z_{t_{n-1}}}(1-e^{-x_nZ_{t_n-t_{n-1}}})\right]\\
&=\int_{\mathbb{R}_+}\mathbb{E}_z[e^{-x_1Z_{t_1}}\cdots Z_{t_{n-1}}e^{-(x_{n-1}+y_{n-1})Z_{t_{n-1}}}]\mathbb{P}_{y_{n-1}}(V_{t_{n}-t_{n-1}}\leq x_n)\ddr y_{n-1}\\
&=\int_{\mathbb{R}_+}\mathbb{E}_z\left[e^{-x_1Z_{t_1}}\cdots \mathbb{E}_{Z_{t_{n-2}}}[Z_{t_{n-1}-t_{n-2}}e^{-(x_{n-1}+y_{n-1})Z_{t_{n-1}-t_{n-2}}}]\right]F_{t_n-t_{n-1}}(y_{n-1},x_n)\ddr y_{n-1}\\
&=\int_{\mathbb{R}^2_+}\!\!\mathbb{E}_z \! \left[e^{-x_1Z_{t_1}}\cdots Z_{t_{n-2}}e^{-Z_{t_{n-2}}y_{n-2}}\right]f_{t_{n-1}-t_{n-2}}(y_{n-2},x_{n-1}+y_{n-1})F_{t_n-t_{n-1}}(y_{n-1},x_n)\ddr y_{n-2}\ddr y_{n-1}\\
&=\int_{\mathbb{R}_+^{n}}\!\!ze^{-zy_0}f_{t_1}(y_0,y_1+x_1)\cdots f_{t_{n-1}-t_{n-2}}(y_{n-2},y_{n-1}+x_{n-1})F_{t_n-t_{n-1}}(y_{n-1},x_n)\ddr y_0 \ddr y_{1}\cdots \ddr y_{n-1}\\
&=\int_{\mathbb{R}_+}\!\! ze^{-zy_0}\mathbb{P}_{y_0}\big(V_{t_1}\geq x_1,\cdots, V_{t_{n-1}}\geq x_{n-1}, V_{t_n}\leq x_n\big),
\end{align*}
where in the fourth equality we use \eqref{derivative}, in the fifth we iterate the argument and in the last line, we use the Markov property of $V$. \qed
\\

\textbf{Acknowledgements} This paper was concluded while V.R was visiting the Department of Statistics at the University of Warwick, United Kingdom; he would like to thank his hosts for partial financial support as well as for their kindness and hospitality. In addition, V.R is grateful for additional financial support from CONAHCyT-Mexico, grant nr. 852367. C.F and A.W are supported by the European Union (ERC, SINGER, 101054787). Views and opinions expressed are however those of the authors only and do not necessarily reflect those of the European Union or the European Research Council. Neither the European Union nor the granting authority can be held responsible for them.

\bibliographystyle{alpha}
\bibliography{sample}

\end{document}